\documentclass{article}

\usepackage{amssymb,amsmath,amscd}
\usepackage[amsmath,thmmarks]{ntheorem}
\theoremseparator{.}
\theoremsymbol{\rule{1ex}{1ex}}
\theorembodyfont{\upshape}
\newtheorem{definition}{Definition}[section]
\newtheorem{remark}[definition]{Remark}

\newtheorem{example}[definition]{Example}

\newtheorem*{proof}{Proof}

\theorembodyfont{\itshape}

\newtheorem{Proposition}[definition]{Proposition}

\newtheorem{Corollary}[definition]{Corollary}
\newtheorem*{nntheorem}{Theorem}
\theoremsymbol{}
\newtheorem{lemma}[definition]{Lemma}
\newtheorem{proposition}[definition]{Proposition}
\newtheorem{theorem}[definition]{Theorem}
\newtheorem{corollary}[definition]{Corollary}

\usepackage[a4paper,hscale=0.8,vscale=0.8,hmarginratio=1:1,includehead,headheight=14pt]{geometry}
\usepackage{fancyhdr}
\usepackage{amssymb,amsmath,amscd,dsfont}
\usepackage[all,cmtip,line]{xy}
\usepackage{slashed}
\newcommand{\op}[1]{\ensuremath{\operatorname{#1}}}
\newcommand{\wt}[1]{\ensuremath{\widetilde{#1}}}
\newcommand{\wh}[1]{\ensuremath{\widehat{#1}}}
\newcommand{\fk}{\ensuremath{\mathfrak{k}}}
\newcommand{\fv}{\ensuremath{\mathfrak{v}}}
\newcommand{\fa}{\ensuremath{\mathfrak{a}}}
\newcommand{\fg}{\ensuremath{\mathfrak{g}}}
\newcommand{\fp}{\ensuremath{\mathfrak{p}}}
\newcommand{\fh}{\ensuremath{\mathfrak{h}}}
\newcommand{\mf}[1]{\ensuremath{\mathfrak{#1}}}
\newcommand{\bC}{\ensuremath{\mathbb{C}}}
\newcommand{\R}{\ensuremath{\mathbb{R}}}
\newcommand{\N}{\ensuremath{\mathbb{N}}}
\newcommand{\Z}{\ensuremath{\mathbb{Z}}}
\newcommand{\id}{\ensuremath{\operatorname{id}}}
\newcommand{\Map}{\ensuremath{\operatorname{Map}}}
\newcommand{\cat}[1]{\ensuremath{\boldsymbol{\operatorname{#1}}}}
\newcommand{\Hom}{\ensuremath{\operatorname{Hom}}}
\newcommand{\im}{\ensuremath{\operatorname{im}}}
\newcommand{\SO}{\ensuremath{\operatorname{SO}}}
\newcommand{\SU}{\ensuremath{\operatorname{SU}}}
\newcommand{\SL}{\ensuremath{\operatorname{SL}}}
\newcommand{\Sp}{\ensuremath{\operatorname{Sp}}}
\newcommand{\se}{\ensuremath{\nobreak\subseteq\nobreak}}
\newcommand{\from}{\ensuremath{\nobreak\colon\nobreak}}
\renewcommand{\to}{\ensuremath{\nobreak\rightarrow\nobreak}}

\usepackage{graphicx}

\usepackage{hyperref}
\hypersetup{
 urlcolor = red,
 colorlinks = true,
 linkcolor = blue,
 citecolor = blue,
 linktocpage = true,
 pdftitle = {The topological group cohomology of Lie groups},
 pdfauthor = {Christoph Wockel},
 bookmarksopen = true,
 bookmarksopenlevel = 1,
 unicode = true,
 hypertexnames =false 
}

\usepackage[all,cmtip,line]{xy} \CompileMatrices

\usepackage[notcite,notref,final]{showkeys}

\usepackage{enumitem}
\setlist[enumerate]{label={\alph*})}

\newcommand{\loc}{\ensuremath{ \op{loc}}}
\newcommand{\Lie}{\ensuremath{ \op{Lie}}}
\newcommand{\dR}{\ensuremath{ \op{dR}}}
\newcommand{\CE}{\ensuremath{ \op{CE}}}
\newcommand{\vE}{\ensuremath{ \op{vE}}}
\newcommand{\Top}{\ensuremath{ \op{top}}}

\newcommand{\SM}{\ensuremath{\op{SM}}}

\newenvironment{tabsection}{}{}

\begin{document}

\begin{flushright}
   {\sf ZMP-HH/14-1}\\
   {\sf Hamburger$\;$Beitr\"age$\;$zur$\;$Mathematik$\;$Nr.$\;$498}\\[2mm]
\end{flushright}

\title{Topological group cohomology of Lie groups and Chern-Weil theory for
compact symmetric spaces} \author{Christoph Wockel}
{\let\newpage\relax\maketitle}

\begin{abstract}
 In this paper we analyse the topological group cohomology of
 finite-dimensional Lie groups. We introduce a technique for computing it (as
 abelian groups) for torus coefficients by the naturally associated long exact
 sequence. The upshot in there is that certain morphisms in this long exact
 coefficient sequence can be accessed (at least for semi-simple Lie groups)
 very conveniently by the Chern-Weil homomorphism of the naturally associated
 compact dual symmetric space. Since the latter is very well-known, this gives
 the possibility to compute the topological group cohomology of the classical
 simple Lie groups. In addition, we establish a relation to characteristic
 classes of flat bundles.
\end{abstract}

\medskip


\medskip


\tableofcontents

\section*{Introduction} %
\label{sec:introduction}

\begin{tabsection}
 Topological group cohomology is the cohomology theory for topological groups
 that incorporates both, the algebraic and the topological structure of a
 topological group $G$ with coefficients in some topological $G$-module $A$.
 There are two obvious guesses for this, which already capture parts of the
 theory in special cases.
 
 The first one is the (singular) cohomology of the classifying space $BG$ of
 $G$. This leads to well-defined cohomology groups $H^{n}_{\pi_{1}(BG)}(BG;A)$
 for discrete coefficient groups $A$ (where $H^{n}_{\pi_{1}(BG)}$ denotes the
 cohomology of the corresponding local coefficient system on $BG$). However, if
 $A$ is non-discrete, then $H^{n}_{\pi_{1}(BG)}(BG;A)$ is not even
 well-defined, since $BG$ is only defined up to homotopy equivalence. Moreover,
 $BG$ is trivial if $G$ is contractible, so no homotopy invariant construction
 on $BG$ could capture for instance the Heisenberg group as a central extension
 of $\R\times \R$ by $U(1)$.
 
 The second obvious guess would be the cohomology of the cochain complex of
 continuos $A$-valued functions (see Section
 \ref{sec:a_recap_of_topological_group_cohmology}). We call this the van Est
 cohomology $H^{n}_{\vE}(G;A)$ of $G$, since it has first been exhaustively
 analysed (in the case of Lie groups) by van Est in the 50's and 60's. However,
 this has a reasonable interpretation as a relative derived functor only in the
 case that $A$ is a topological vector space
 \cite{HochschildMostow62Cohomology-of-Lie-groups}, and captures in this
 respect the case that is contrary to the case of discrete coefficients.
 
 The topological group cohomology interpolates between these two extreme case.
 It has first been defined by Segal and Mitchison in
 \cite{Segal70Cohomology-of-topological-groups} (see also
 \cite{Deligne74Theorie-de-Hodge.-III,Moore76Group-extensions-and-cohomology-for-locally-compact-groups.-III,Cattaneo77On-locally-continuous-cocycles,Brylinski00Differentiable-Cohomology-of-Gauge-Groups,Flach08Cohomology-of-topological-groups-with-applications-to-the-Weil-group})
 and recently been put into a unifying framework in
 \cite{WagemannWockel13A-Cocycle-Model-for-Topological-and-Lie-Group-Cohomology}.
 We denote the corresponding cohomology groups by $H^{n}(G;A)$ with no
 additional subscript. If $A=\fa/\Gamma$ for some contractible $G$-module $\fa$
 and some submodule $\Gamma$, then the topological group cohomology
 interpolates between the classifying space cohomology and the van Est
 cohomology in the sense that there is a long exact sequence
 \begin{equation}\label{eqn:intro_2}
  \cdots\to H^{n-1}(G;A)\to H^{n}_{\pi_{1}(BG)}(BG;\Gamma)\to H^{n}_{\vE}(G;\fa)\to 
  H^{n}(G;A)\to H^{n+1}_{\pi_{1}(BG)}(BG;\Gamma)\to\cdots
 \end{equation}
 (see Section \ref{sec:the_long_exact_sequence}).
 
 The bulk of this paper is devoted to analyse this exact sequence in the case
 that $G$ is a Lie groups and that the coefficients are smooth. In particular,
 we establish a connection to Lie algebra cohomology that we then exploit in
 the sequel to calculate certain important morphisms of the above sequence in
 explicit terms. This then permits to calculate $H^{n}(G;U(1))$ for some (in
 principle all) semi-simple Lie groups in terms of the (well-known) Chern-Weil
 homomorphism of compact symmetric spaces. Moreover, we establish a connection
 to characteristic classes of flat bundles.
 
 \medskip
 
 We now shortly list the results of the individual sections. Section
 \ref{sec:a_recap_of_topological_group_cohmology} recalls the basic facts about
 topological group cohomology. In Section \ref{sec:the_long_exact_sequence} we
 introduce the long exact sequence \eqref{eqn:intro_2} and reinterpret it in
 terms of relative group cohomology. In particular, we will motivate why it is
 natural to think of the morphisms
 \begin{equation}\label{eqn:intro_1}
  \varepsilon^{n}\from  H^{n}_{\pi_{1}(BG)}(BG;\Gamma)\to H^{n}_{\vE}(G;\fa)
 \end{equation}
 as connecting morphisms (instead of
 $H^{n-1}(G;A)\to H^{n}_{\pi_{1}(BG)}(BG;\Gamma)$). Since these morphisms play
 a distinguished r\^ole for the whole theory we call them \emph{characteristic
 morphisms}\footnote{See Remark \ref{rem:flat_characteristic_classes} for the
 interpretation in terms of flat characteristic classes and the relation to the
 characteristic morphisms $H^{n}_{\Lie}((\fg,K);\R)\to H^{n}_{\op{gp}}(G,\R)$
 from \cite{Morita01Geometry-of-characteristic-classes}.}. Note that both sides
 of \eqref{eqn:intro_1} are well-known in many cases, so the question arises
 whether $\varepsilon^{n}$ also has a known interpretation.
 
 In Section \ref{sec:the_relation_to_relative_lie_algebra_cohomology} we
 establish the relation to Lie algebra cohomology. Those cohomology classes
 which have trivial Lie algebra cohomology classes have a natural
 interpretation as flat bundles (or higher bundles, such as bundle gerbes).
 This gives in particular rise to the interpretation of the image of
 $\varepsilon^{n}$ as flat characteristic classes.
 
 Section \ref{sec:subalgebras_non_cohomologous_to_zero} then treats the case in
 which all characteristic morphisms vanish. This condition can be checked very
 conveniently for semi-simple Lie groups, since there it can be read off the
 associated compact dual $G_{u}/K$ of the non-compact symmetric space $G/K$
 naturally associated to $G$ (where $K\leq G$ is a maximal compact subgroup).
 In the case that all characteristic morphisms vanish the cohomology groups
 $H^{n}(G;U(1))$ may be computed as
 \begin{equation*}
  H^{n}(G;\fa/\Gamma)\cong H^{n}_{\Lie}((\fg,K),\fa)\oplus H^{n+1}_{\pi_{1}(BG)}(BG;\Gamma).
 \end{equation*}
 
 In Section \ref{sec:semi_simple_lie_groups} we then show that the
 characteristic homomorphisms $\varepsilon^{n}$ can be computed in terms of the
 compact dual $G_{u}/K$. More precisely, let $f\from G_{u}/K\to BK$ be a
 classifying map for the principal $K$-bundle $G_{u}\to G_{u}/K$ and let
 $j\from \Gamma\to \fa$ denote the inclusion. The main result of Section
 \ref{sec:semi_simple_lie_groups} is then the following
 
 \begin{nntheorem}
  Suppose $G$ is a semi-simple Lie group that acts trivially on $\fa$ and
  suppose $\Gamma\leq \fa$ is discrete. Then there exist isomorphisms
  $ H^{n}(G;\Gamma)\xrightarrow{~\cong~}H^{n}_{\Top}(BK;\Gamma)$ and
  $H^{n}(G;\fa)\xrightarrow{~\cong~}H^{n}_{\Top}(G_{u}/K;\fa)$ such that the
  diagram
  \begin{equation*}
   \xymatrix{
   H^{n}(G;\Gamma)\ar[rr]^{\varepsilon^{n}}\ar[d]^{\cong}&&H^{n}(G;\fa)\ar[d]^{\cong}\\
   H^{n}_{\Top}(BK;\Gamma)\ar[r]^{j_{*}}&H^{n}_{\Top}(BK;\fa)\ar[r]^{f^{*}}&H^{n}_{\Top}(G_{u}/K;\fa)
   }
  \end{equation*}
  commutes.
 \end{nntheorem}
 Since $j_{*}$ and $f^{*}$ can be computed explicitly by the Chern-Weil
 homomorphism of $G_{u}\to G_{u}/K$, the preceding theorem gives a very good
 control on the long exact sequence \eqref{eqn:intro_2}. In particular, it
 gives a good control on which classes in $H^{n}(G;\fa)$ are flat.
 
 In Section \ref{sec:examples} we then treat some examples. One of the perhaps
 most interesting consequences is that $\varepsilon^{2q}$ does not vanish on
 the Euler class
 $E_{q}\in H^{2q}(\SL_{2q}(\R);\Z)\cong H^{2q}_{\Top}(B\SO_{2q};\Z)$, and thus
 $\varepsilon^{2q}(E_{q})$ yields a flat characteristic class.
 
 \medskip
 
 This brings us to an analysis of the results obtained in this paper. At first,
 the computational results obtained in Section \ref{sec:examples} and the
 connection to the Chern-Weil homomorphism of compact symmetric spaces is new.
 The flatness of the Euler class is of course not new (cf.\
 \cite{Milnor58On-the-existence-of-a-connection-with-curvature-zero,Dupont78Curvature-and-characteristic-classes}),
 but what is new is the perspective on this phenomena that topological group
 cohomology yields. Moreover, the perspective to the Chern-Weil homomorphism as
 a push-forward in coefficients within the same theory seems to be new. In
 particular, this gives a conceptual interpretation of many seemingly ad-hoc
 relations between the cohomology of classifying spaces and the van Est
 cohomology (see
 \cite{08Guidos-book-of-conjectures,Karlsson04Characteristic-Classes-and-Bounded-Cohomology,Dupont79Bounds-for-characteristic-numbers-of-flat-bundles,Dupont76Simplicial-de-Rham-cohomology-and-characteristic-classes-of-flat-bundles}).
 
 Since the aforementioned flatness of the Euler class and the relation between
 the cohomology of classifying spaces and the van Est cohomology all occur
 naturally in the context of bounded continuous cohomology (see also
 \cite{Monod06An-invitation-to-bounded-cohomology,Monod01Continuous-bounded-cohomology-of-locally-compact-groups}),
 this suggest that there is a close relation between topological group
 cohomology and bounded continuous cohomology. We expect that a further
 analysis of the techniques presented in this paper might also lead to new
 applications and insights there.
\end{tabsection}

\section*{Acknowledgements}

The author thanks the Max Planck Institute for Mathematics in Bonn for
hospitality, where a major part of the research of the present paper was done.
Moreover, we thank Georgios Chatzigiannis for the initial discussions that lead
to the motivating questions for the paper and for proofreading early versions
of it. Thanks go also to Friedrich Wagemann for intensive discussions about the
principal topics of the paper. We also thank Thomas Nikolaus and Jim Stasheff
for some enlightening discussions about the r\^ole of classifying spaces and
various other cohomology theories associated to them and the relation to
topological group cohomology.

\section{A recap of topological group cohomology} %
\label{sec:a_recap_of_topological_group_cohmology}

\begin{tabsection}
 The purpose of this section is to fix notation and to introduce concepts. More
 detailed expositions can be found in
 \cite{WagemannWockel13A-Cocycle-Model-for-Topological-and-Lie-Group-Cohomology,Guichardet80Cohomologie-des-groupes-topologiques-et-des-algebres-de-Lie,Segal70Cohomology-of-topological-groups,HochschildMostow62Cohomology-of-Lie-groups}.
 Throughout this section, $G$ is an arbitrary topological group\footnote{With
 this we mean a group object in the category of compactly generated Hausdorff
 spaces, i.e., we endow products with the compactly generated product
 topology.} and $A$ a topological $G$-module. By a topological $G$-module we
 mean a locally contractible topological abelian group $A$ that is a $G$-module
 such that the action map $G\times A\to A$ is continuous. A short exact
 sequence $A\to B\to C$ of topological $G$-modules is defined to be a sequence
 of topological $G$-modules such that $B$ is a principal $A$-bundle over $C$.
 The sequence $A\to B\xrightarrow{~q~} C$ is said to be topologically trivial
 if the principal bundle is trivial, i.e., if there exists $\sigma\from C\to B$
 continuous such that $q \circ \sigma=\id_{C}$. Moreover, let $G$ act
 continuously from the left on a space $X$ (in case $X=G$ we will always
 consider the action by left multiplication). Then we endow $\Map(X,A)$
 (arbitrary set maps for the moment) with the left action of $G$ given by
 $(g.f)(x):=g.(f(g^{-1}.x))$.\\
 
 Now there are several cohomology groups associated to this setting:
 
 \begin{enumerate}
  \item The \emph{van Est cohomology}
        \begin{equation*}
         H^{n}_{\vE}(X;A):=H^{n}(C^{0}_{\vE}(X,A)^{G}\xrightarrow{~d~} C^{1}_{\vE}(X,A)^{G}\xrightarrow{~d~}\cdots)
        \end{equation*}
        with
        \begin{equation}\label{eqn:AS_differential}
         C^{n}_{\vE}(X,A):=C(X^{n+1},A)\quad\text{ and }\quad  
         df(g_{0},...,g_{n+1}):=\sum_{i=0}^{n+1}(-1)^{i}f(g_{0},...,\wh{g_{i}},...,g_{n+1}).
        \end{equation}
        If $X=G$, then we obtain the van Est cohomology $H^{n}_{\vE}(G;A)$
        (which is called $H^{n}_{\op{glob,c}}(G;A)$ in
        \cite{WagemannWockel13A-Cocycle-Model-for-Topological-and-Lie-Group-Cohomology}\footnote{To
        match up with
        \cite{WagemannWockel13A-Cocycle-Model-for-Topological-and-Lie-Group-Cohomology}
        one has to pass from the homogeneous cochain complex to the
        inhomogeneous one, i.e., identify $\Map(G^{n+1},A)^{G}$ with
        $\Map(G^{n},A)$ via $f\mapsto F$ with
        $F(g_{1},...,g_{n}):=F(1,g_{1},g_{1} g_{2},...,g_{1}\cdots g_{n})$ (see
        \cite[Section I.5]{Brown94Cohomology-of-groups},
        \cite[n\textsuperscript{o}
        I.3.1]{Guichardet80Cohomologie-des-groupes-topologiques-et-des-algebres-de-Lie}
        or \cite[Appendix
        B]{Neeb04Abelian-extensions-of-infinite-dimensional-Lie-groups}).}).
        If, moreover, $G$ is a Lie group, $A$ is a smooth $G$-module, $X$ is a
        manifold and the action is smooth, then we also have the corresponding
        smooth version
        \begin{equation*}
         H^{n}_{\vE,s}(X;A):=H^{n}(C^{0}_{\vE,s}(X;A)^{G}\xrightarrow{~d~}
         C^{1}_{\vE,s}(X,A))^{G}\xrightarrow{~d~}\cdots)
        \end{equation*}
        with $C^{n}_{\vE,s}(X,A):=C^{\infty}(X^{n+1},A)$\footnote{In the smooth
        category we endow products with the usual product smooth structure}.
        Note that if $G$ is a finite-dimensional Lie group and $A=\fa$ is a
        smooth and quasi-complete locally convex $G$-module, then by
        \cite[Theorem 5.1]{HochschildMostow62Cohomology-of-Lie-groups} the
        inclusion $C^{n}_{\vE,s}(G,\fa)\hookrightarrow C^{n}_{\vE}(G,\fa)$
        induces an isomorphism $H^{n}_{\vE,s}(G;\fa)\cong H^{n}_{\vE}(G;\fa)$.
  \item \label{item:def_SM} The \emph{Segal-Mitchison cohomology} (for
        simplicity we only consider the case $X=G$)
        \begin{equation*}
         H^{n}_{\SM}(G;A):=H^{n}(C(G,EA)^{G}\xrightarrow{~d~}C(G,B_{G}A)^{G}\xrightarrow{~d~}\cdots),
        \end{equation*}
        where $B_{G}A:=C(G,EA)/A$ and $EA$ is a chosen model for the universal
        bundle of the topological abelian group $A$ such that $EA\to BA$ admits
        a local section \cite[Appendix
        A]{Segal70Cohomology-of-topological-groups}. If $A$ is contractible,
        then we may assume that $EA=A$ and thus one sees that
        $H^{n}_{\SM}(G;A)\cong H^{n}_{\vE}(G;A)$ in this case \cite[Proposition
        3.1]{Segal70Cohomology-of-topological-groups}. On the other hand, if
        $A=A^{\delta}$ is discrete, then \cite[Proposition
        3.3]{Segal70Cohomology-of-topological-groups} shows that
        $H^{n}_{\SM}(G;A)\cong H^{n}_{\pi_{1}(BG)}(BG;A)$ (where
        $BG:=|BG_{\bullet}|$ is the classifying space of $G$ and
        $H^{n}_{\pi_{1}(BG)}$ denotes the sheaf cohomology of the local system
        of the $\pi_{1}(BG)\cong \pi_{0}(G)$-action on the discrete group $A$).
  \item The \emph{locally continuous cohomology}
        \begin{equation*}
         H^{n}_{\loc}(X;A):=H^{n}(C^{0}_{\loc}(X,A)^{G}\xrightarrow{~d~} C^{1}_{\loc}(X,A)^{G}\xrightarrow{~d~}\cdots ),
        \end{equation*}
        where
        \begin{equation*}
         C^{n}_{\loc}(X,A):=\{f\from X^{n+1}\to A\mid f\text{ is continuous on some neighbourhood of the diagonal }\Delta^{n+1}X \}.
        \end{equation*}
        By abuse of notation we sometimes refer to the elements of
        $C^{n}_{\loc}(X,A)$ as locally continuous maps or cochains. Again, if
        $G=X$, then we obtain the locally continuous cohomology
        $H^{n}_{\loc}(G;A)$. We have a natural morphism
        $H^{n}_{\vE}(X;A)\to H^{n}_{\loc}(X;A)$ induced from the inclusion
        $C^{n}_{\vE}(X,A)\hookrightarrow C^{n}_{\loc}(X,A)$. Note that this is
        an isomorphism if either $X$ is contractible \cite[Theorem
        5.16]{Fuchssteiner11A-Spectral-Sequence-Connecting-Continuous-With-Locally-Continuous-Group-Cohomology}
        or if $X=G$ is metrisable and $A$ is contractible by group
        homomorphisms \cite[Proposition
        3.6]{FuchssteinerWockel11Topological-Group-Cohomology-with-Loop-Contractible-Coefficients}.
        
        If, moreover, $G$ is a Lie group, $A$ is a smooth $G$-module, $X$ is a
        manifold and the action is smooth, then we also have the corresponding
        smooth version
        \begin{equation*}
         H^{n}_{\loc,s}(X;A):=H^{n}(C^{0}_{\loc,s}(X,A)^{G}\xrightarrow{~d~} C^{1}_{\loc,s}(X,A)^{G}\xrightarrow{~d~}\cdots ),
        \end{equation*}
        where
        \begin{equation*}
         C^{n}_{\loc,s}(X,A):=\{f\from X^{n+1}\to A\mid f\text{ is smooth on some neighbourhood of the diagonal } \Delta^{n+1}X\}.
        \end{equation*}
        By abuse of notation we sometimes refer to the elements of
        $C^{n}_{\loc,s}(X,A)$ as locally smooth maps or cochains. Again, if
        $G=X$, then we obtain locally smooth cohomology $H^{n}_{\loc,s}(G;A)$
        of $G$ considered in
        \cite{WagemannWockel13A-Cocycle-Model-for-Topological-and-Lie-Group-Cohomology}.
        If we assume, furthermore, that $G$ is finite-dimensional and $\fa$ is
        quasi-complete, then the inclusion
        $C^{n}_{\loc,s}(G,A)\hookrightarrow C^{n}_{\loc}(G,A)$ is a
        quasi-isomorphism, i.e., induces an isomorphism in cohomology
        $H^{n}_{\loc,s}(G;A)\cong H^{n}_{\loc}(G;A)$ \cite[Proposition
        I.7]{WagemannWockel13A-Cocycle-Model-for-Topological-and-Lie-Group-Cohomology}.
        We will often identify $H^{n}_{\loc,s}(G;A)$ with $H^{n}_{\loc}(G;A)$
        via this identification.
 \end{enumerate}
 
 If $A\to B\to C$ is a short exact sequence of topological $G$-modules, then we
 have long exact sequences
 \begin{equation*}
  \cdots \to   H^{n-1}_{\SM}(G;C)\xrightarrow{\delta^{n-1}}  H^{n}_{\SM}(G;A)\to  H^{n}_{\SM}(G;B)\to   H^{n}_{\SM}(G;C)\xrightarrow{\delta^{n}}   H^{n}_{\SM}(G;A)\to\cdots
 \end{equation*}
 and
 \begin{equation*}
  \cdots \to   H^{n-1}_{\loc}(G;C)\xrightarrow{\delta^{n-1}}  H^{n}_{\loc}(G;A)\to  H^{n}_{\loc}(G;B)\to   H^{n}_{\loc}(G;C)\xrightarrow{\delta^{n}}   H^{n}_{\loc}(G;A)\to\cdots
 \end{equation*}
 (cf.\ \cite[Proposition 2.3]{Segal70Cohomology-of-topological-groups} and
 \cite[Remark
 I.2]{WagemannWockel13A-Cocycle-Model-for-Topological-and-Lie-Group-Cohomology}).
 These long exact sequences are natural with respect to morphisms of short
 exact sequences (cf.\ \cite[Section
 VI]{WagemannWockel13A-Cocycle-Model-for-Topological-and-Lie-Group-Cohomology}).
 Since $H^{n}_{\SM}(G;A)$ and $H^{n}_{\loc}(G;A)$ coincide for (loop)
 contractible $A$ with $H^{n}_{\vE}(G;A)$ (cf.\ \cite[Proposition
 3.1]{Segal70Cohomology-of-topological-groups} and
 \cite{FuchssteinerWockel11Topological-Group-Cohomology-with-Loop-Contractible-Coefficients}),
 this implies that we have isomorphisms of $\delta$-functors (cf.\
 \cite[Section
 VI]{WagemannWockel13A-Cocycle-Model-for-Topological-and-Lie-Group-Cohomology})
 \begin{equation}\label{eqn:isomorphism_btw_SM_and_loc}
  H^{n}_{\SM}(G;A)\cong H^{n}_{\loc}(G;A)
 \end{equation}
 (under the additional assumption that the product topology on $G^{n}$ is
 compactly generated, see \cite[Corollary
 IV.8]{WagemannWockel13A-Cocycle-Model-for-Topological-and-Lie-Group-Cohomology}).
 The same argument shows that the Segal-Mitchison and the locally continuous
 cohomology coincides (under some mild additional assumptions) with many other
 cohomology theories for topological groups, as for instance the simplicial
 group cohomology from
 \cite{Deligne74Theorie-de-Hodge.-III,Brylinski00Differentiable-Cohomology-of-Gauge-Groups}
 (see \cite[Corollary
 IV.7]{WagemannWockel13A-Cocycle-Model-for-Topological-and-Lie-Group-Cohomology}),
 the measurable group cohomology from
 \cite{Moore76Group-extensions-and-cohomology-for-locally-compact-groups.-III}
 (see \cite[Remark
 IV.13]{WagemannWockel13A-Cocycle-Model-for-Topological-and-Lie-Group-Cohomology})
 and the cohomology groups from
 \cite{Flach08Cohomology-of-topological-groups-with-applications-to-the-Weil-group}
 (see \cite[Remark
 IV.12]{WagemannWockel13A-Cocycle-Model-for-Topological-and-Lie-Group-Cohomology}).
 We believe that this is the ``correct'' notion of a cohomology theory for
 topological groups and thus call it the \emph{topological group cohomology}.
 In case that we do not refer to a specific cocycle model we will just denote
 it by $H^{n}(G;A)$.\\
 
 Note that the argument leading to the isomorphism
 $H^{n}_{\SM}(G;A)\cong H^{n}_{\loc}(G;A)$ does \emph{not} show that the
 topological group cohomology is isomorphic to the van Est cohomology, since for
 the van Est cohomology we only have a long exact sequence
 \begin{equation*}
  \cdots \to   H^{n-1}_{\vE}(G;C)\xrightarrow{\delta^{n-1}}  H^{n}_{\vE}(G;A)\to  H^{n}_{\vE}(G;B)\to   H^{n}_{\vE}(G;C)\xrightarrow{\delta^{n}}   H^{n}_{\vE}(G;A)\to\cdots
 \end{equation*}
 if the short exact sequence $A\to B\to C$ is \emph{topologically trivial}.
\end{tabsection}

\begin{remark}
 The functors $H^{n}_{\vE}$, $H^{n}_{\SM}$ and $H^{n}_{\loc}$ are also natural
 in the first argument in the sense that a continuous morphism
 $\varphi\from H\to G$ induces morphisms
 $\varphi^{*}\from H^{n}(G;A)\to H^{n}(H;\varphi^{*}A)$, where $\varphi^{*}A$
 denotes the pull-back module. Indeed, $\varphi$ induces morphisms of cochain
 complexes
 \begin{equation*}
  \varphi^{*}\from C(G^{n+1},A)^{G}\to C(H^{n+1},\varphi^{*}A)^{H}
  \quad \text{ and }\quad
  \varphi^{*}\from C_{\loc}^{n}(G,A)^{G}\to C(H,\varphi^{*}A)^{H},\quad f \mapsto f \circ \varphi,
 \end{equation*}
 which induce morphisms
 $\varphi^{*}\from H^{n}_{\vE}(G;A)\to H^{n}_{\vE}(H;\varphi^{*}A)$ and
 $\varphi^{*}\from H^{n}_{\loc}(G;A)\to H^{n}_{\loc}(H;\varphi^{*}A)$ in
 cohomology.
 
 The morphisms for $H^{n}_{\SM}$ are induced as follows. Recall from
 \cite[Appendix A]{Segal70Cohomology-of-topological-groups} that the module
 structure on $EA$ is induced from the action of the simplicial topological
 group $G_{\bullet}$ (i.e., $G_{n}:=G$) on the abelian simplicial topological
 group $EA_{\bullet}$ (i.e., $EA_{n}:=A^{n+1}$) via the diagonal action of $G$
 on $A^{n+1}$. In the geometric realisation we thus obtain a map
 \begin{equation*}
  G\times EA = |G_{\bullet}|\times |EA_{\bullet}| \cong  |G_{\bullet}\times EA_{\bullet}|\to |EA_{\bullet}|=EA
 \end{equation*}
 defining the module structure. From this one sees that we have
 $\varphi^{*}EA=E \varphi^{*}A$. We now observe that a morphism
 $\psi\from \varphi^{*} A\to C$ induces a morphism
 \begin{equation*}
  E_{\varphi,\psi}\from \varphi^{*}  C(G,EA)\to C(H,EC),\quad
  f \mapsto E(\psi) \circ f \circ \varphi.
 \end{equation*}
 Since $E_{\varphi,\psi}$ preserves the constant maps it induces a morphism
 \begin{equation*}
  B_{\varphi,\psi}\from \varphi^{*} B_{G}A \to B_{H}C
 \end{equation*}
 Inductively we obtain morphisms
 $B_{\varphi,\psi}^{n}\from \varphi^{*} B_{G}^{n}A\to B_{H}^{n}C$ and thus
 morphisms
 \begin{equation*}
  E_{\varphi,B_{\varphi,\psi}^{n}}\from \varphi^{*}  C(G,EB_{G}^{n}A)\to C(H,EB_{H}^{n}C),\quad f\mapsto E(B_{\varphi,\psi}^{n}) \circ f \circ \varphi.
 \end{equation*}
 In particular, if we set $C=\varphi^{*}A$, then we have morphisms of cochain
 complexes
 \begin{equation*}
  \xymatrix{  
  \cdots\ar[r] & C(G,EB_{G}^{n}A)^{G}\ar[r]\ar[d] & C(G,EB_{G}^{n+1}A)^{G} \ar[r]\ar[d] & \cdots\\
  \cdots\ar[r] & C(H,EB_{H}^{n}\varphi^{*} A)^{H}\ar[r] & C(H,EB_{H}^{n+1}\varphi^{*} A)^{H} \ar[r] & \cdots
  }
 \end{equation*}
 that induce morphisms
 $\varphi^{*}\from H^{n}_{\SM}(G;A)\to H^{n}_{\SM}(H;\varphi^{*}A)$ in
 cohomology.
 
 Obviously, if $\alpha\from A\to D$ is a morphism of topological $G$-modules,
 then we get a morphism $\varphi^{*}\alpha\from \varphi^{*}A\to \varphi^{*}D$
 and the diagram
 \begin{equation*}
  \xymatrix{
  H^{n}(G;A)\ar[rr]^{\alpha_{*}} \ar[d]^{\varphi^{*}} && H^{n}(G;D)\ar[d]^{\varphi^{*}}\\
  H^{n}(H;\varphi^{*} A)\ar[rr]^{\alpha_{*}}  && H^{n}(H;\varphi^{*}D)\\
  }
 \end{equation*}
 commutes.
\end{remark}

\begin{proposition}\label{prop:isomorphism_btw_SM_and_loc_is_natural_in_first_argument}
 The isomorphisms $H^{n}_{\SM}(G;A)\cong H^{n}_{\loc}(G;A)$ from
 \eqref{eqn:isomorphism_btw_SM_and_loc} are natural with respect to
 $\varphi^{*}$, i.e., if $\varphi\from H\to G$ is a morphisms of topological
 groups, then the diagram
 \begin{equation}\label{eqn:isomorphism_btw_SM_and_loc_is_natural_in_first_argument}
  \xymatrix{
  H^{n}_{\SM}(G;A)\ar[r]^{\cong}\ar[d]^{\varphi^{*}} & H^{n}_{\loc}(G;A)\ar[d]^{\varphi^{*}}\\
  H^{n}_{\SM}(H;\varphi^{*}A) \ar[r]^{\cong} & H^{n}_{\loc}(H;\varphi^{*}A)
  }
 \end{equation}
 commutes for each $n\in\N_{0}$.
\end{proposition}

\begin{proof}
 We have the following $\delta$-functors (cf.\ \cite[Section
 VI]{WagemannWockel13A-Cocycle-Model-for-Topological-and-Lie-Group-Cohomology})
 \begin{equation*}
  \cat{G-Mod}\to\cat{Ab},\quad
  A\mapsto H^{n}_{\SM}(G;A),\quad A\mapsto H^{n}_{\SM}(H;\varphi^{*}A),\quad A\mapsto H^{n}_{\loc}(G;A),\quad 
  A\mapsto H^{n}_{\loc}(H;\varphi^{*}A).
 \end{equation*}
 Observe that $\varphi^{*}$ constitute morphisms of $\delta$-functors. Since
 \begin{equation*}
  H^{n}_{\SM}(G;E_{G}B_{G}^{n}A)=H^{n}_{\loc}(G;E_{G}B_{G}^{n}A)=0
 \end{equation*}
 it suffices by \cite[Theorem
 VI.2]{WagemannWockel13A-Cocycle-Model-for-Topological-and-Lie-Group-Cohomology}
 to observe that
 \eqref{eqn:isomorphism_btw_SM_and_loc_is_natural_in_first_argument} commutes
 for $n=0$. The latter is trivial.
\end{proof}

\begin{remark}
 One relation that we obtain from the above is in the case that the $G$-action
 is also continuous for the discrete topology $A^{\delta}$ on $A$ (this happens
 for instance if $G$ is locally contractible and $G_{0}$ acts trivially). Then
 we have isomorphisms
 \begin{equation*}
  \zeta^{n}\from H^{n}_{\pi_{1}(BG)}(BG;A^{\delta})\xrightarrow{~\cong~}H^{n}_{SM}(G;A^{\delta})
 \end{equation*}
 (cf.\ Remark in \S 3 of \cite{Segal70Cohomology-of-topological-groups}). If we
 identify $H^{n}_{SM}(G;A^{\delta})$ with $H^{n}_{\loc}(G;A^{\delta})$, then
 the morphism $A^{\delta}\to A$ induces a morphism
 \begin{equation*}
  \flat^{n}\from  H^{n}_{\pi_{1}(BG)}(BG;A)\to H^{n}_{\loc}(G;A),
 \end{equation*}
 which is of course an isomorphism if $A$ already is discrete. On the other
 hand, if $G$ is discrete, then $\flat^{n}$ is the well-known isomorphism
 $H^{n}_{\op{gp}}(G;A)\cong H^{n}_{\pi_{1}(BG)}(BG;A)$
 \cite{Brown94Cohomology-of-groups}. Here $H^{n}_{\op{gp}}(G;A)$ is the group
 cohomology of the abstract group $G$ with coefficients in $A$, which coincides
 (literally at the cochain level) with $H^{n}_{\loc}(G^{\delta};A^{\delta})$.
 From the explicit description of $\zeta^{n}$ it follows that $\zeta^{n}$ and
 $\flat^{n}$ are natural with respect to morphisms of groups and of
 coefficients, i.e., if $\varphi\from H\to G$ is a morphism of topological
 groups and $\alpha\from A\to D$ is a morphism of topological $G$-modules, then
 the diagrams
 \begin{equation*}
  \vcenter{  \xymatrix{
  H^{n}_{\pi_{1}(BG)}(BG;A)\ar[d]^{\alpha_{*}}\ar[r]^(.58){{\flat}^{n}} & H^{n}_{\loc}(G;A)\ar[d]^{\alpha_{*}}\\
  H^{n}_{\pi_{1}(BG)}(BG;D)\ar[r]^(.58){{\flat}^{n}} & H^{n}_{\loc}(G;D)
  }}\quad\text{ and }\quad
  \vcenter{  \xymatrix{
  H^{n}_{\pi_{1}(BG)}(BG;A)\ar[d]^{\varphi^{*}}\ar[r]^(.58){{\flat}^{n}} & H^{n}_{\loc}(G;A)\ar[d]^{\varphi^{*}}\\
  H^{n}_{\pi_{1}(BH)}(BH;\varphi^{*}A)\ar[r]^(.58){{\flat}^{n}} & H^{n}_{\loc}(H;A)
  }}
 \end{equation*}
 commute (and likewise for $\zeta^{n}$). If $G$ is a finite-dimensional Lie
 group and $A=\fa/\Gamma$ for $\fa$ locally convex and quasi-complete, then we
 may interpret $\flat^{n}$ as a natural morphism to $H^{n}_{\loc,s}(G;A)$ via
 the identification $H^{n}_{\loc}(G;A)\cong H^{n}_{\loc,s}(G;A)$.
\end{remark}

We will follow the convention that we denote morphism in cohomology that are
induced by morphisms of groups, spaces or coefficient modules by upper and
lower stars. Morphisms that are induced by manipulations of cochains will be
denoted by the corresponding cohomology index. If we use the upper star as the
index of cohomology groups when referring to the whole cohomology algebra,
instead to a single abelian group in one specific degree (and a plus there
refers to the cohomology in positive degree). For the convenience of the
reader, we collect here the definitions of the cohomology groups and some of
their chain complexes that we will use throughout (in the order in which they
appear in the text):

\medskip

\begin{tabular}{r|r|l}
 $H^{n}_{\vE}(X,A)$ & $C^{n}_{\vE}(X,A)^{G}$ & $C(X^{n+1},A)^{G}$\\[.2em]
 $H^{n}_{\vE,s}(X,A)$ &	$C^{n}_{\vE,s}(X,A)^{G}$ & $C^{\infty}(X^{n+1},A)^{G}$\\[.2em]
 $H^{n}_{\SM}(G,A)$ && see \ref{item:def_SM} above\\[.2em]
 $H^{n}_{\pi_{1}(BG)}(BG;A)$ & & see \ref{item:def_SM} above \\[.2em]
 $H^{n}_{\loc}(X,A)$ &	$C^{n}_{\loc}(X,A)^{G}$ & $\{f\from X^{n+1}\to A\mid f\text{ is cont.\ on some neighbh.\ of }\Delta^{n+1}X\}^{G}$\\[.2em]
 $H^{n}_{\loc,s}(X,A)$ &	$C^{n}_{\loc,s}(X,A)^{G}$ & $\{f\from X^{n+1}\to A\mid f\text{ is smooth on some neighbh.\ of }\Delta^{n+1}X\}^{G}$\\[.2em]
$H^{n}(G;A)$ & & any of $H^{n}_{\SM}(G;A)$, $H^{n}_{\loc}(G;A)$ (or $H^{n}_{\loc,s}(G;A)$)\\[.2em]
$H^{n}_{\Top}(X;A)$ & & cohomology of the (underlying) topological space $X$ with coeff. \\ &&~\hfill in the (abstract) abelian group $A$\\[.2em]
 $H^{n}_{\Lie}((\fg,\fk);\fa)$ &	$C_{\CE}^{n}((\fg,\fk),\fa)$ & $\Hom(\Lambda^{n}\fg/\fk,\fa)^{\fk}$\\[.2em]
 $H^{n}_{\Lie}((\fg,K);\fa)$ &  $C_{\CE}^{n}((\fg,K),\fa)$ & $\Hom(\Lambda^{n}\fg/\fk,\fa)^{K}$\\[.2em]
 $H^{n}_{\loc,s}(G;A)$ &  $\wt{C}_{\loc,s}^{n}(G,A)$ & $\{f\from G^{n}\to A\mid f\text{ is smooth on some id. neighbh.}\}$\\[.2em]
 $H^{n}_{\loc,s}(G;A)$ &  $\wt{C}_{\loc,s}^{0,n}(G,A)$ & $\{f\in \wt{C}_{\loc,s}^{n}(G,A)\mid
  f(g_{1},...,g_{n})=0 \text{ if }g_{i}=1\text{ for some }i\}$\\[.2em]
 $H^{n}_{\loc,s}((G,K);A)$ &	$\wt{C}_{\loc,s}((G,K),A)$ & $\{f\in C^{n}_{\loc,s}(G,A)\mid
  f(g_{1},...,g_{n})=k_{0}.f(k_{0}^{-1}g_{1}k_{1},...,k_{n-1}^{-1}g_{n}k_{n})$\\
 & &~\hfill$\text{ for all }g_{1},...,g_{n}\in G,k_{0},...,k_{n}\in K\}$

\end{tabular}

\section{The long exact sequence and the characteristic morphisms} %
\label{sec:the_long_exact_sequence}

\begin{tabsection}
 In this section we analyse the long exact sequence in topological group
 cohomology for torus (or more generally $K(\Gamma,1)$) coefficients. We will
 try to motivate why it is a good idea to look at this sequence. The general
 assumptions in this section are as in the previous one. We only assume, in
 addition, that the coefficient module is $A=\fa/\Gamma$ for some contractible
 $G$-module $\fa$ and a discrete submodule $\Gamma\leq\fa$. Recall that
 $H^{n}(G;A)$ always refers to the topological group cohomology of $G$ with
 coefficients in $A$, which can be realised by the models $H^{n}_{\SM}(G;A)$ or
 $H^{n}_{\loc}(G;A)$.
\end{tabsection}
 
\begin{remark}\label{rem:long_exact_sequence}
 The exact coefficient sequence $\Gamma\to \fa\to A$ induces a long exact
 sequence
 \begin{equation}\label{eqn:long_exact_sequence_1}
  \cdots\to H^{n-1}(G;A)\to H^{n}(G;\Gamma)\to H^{n}(G;\fa)\to H^{n}(G;A)\to H^{n+1}(G;\Gamma)\to\cdots.
 \end{equation}
 As described in the previous section we have isomorphisms
 \begin{equation*}
  H^{n}(G;\fa)\cong H^{n}_{\vE}(G;\fa)
  \quad\text{ and }\quad  H^{n}(G;\Gamma)\cong H^{n}_{\pi_{1}(BG)}(BG;\Gamma).
 \end{equation*}
 This leads to a long exact sequence
 \begin{equation}\label{eqn:long_exact_sequence_2}
  \cdots\to H^{n-1}(G;A)\to H^{n}_{\pi_{1}(BG)}(BG;\Gamma)\to H^{n}_{\vE}(G;\fa)\to 
  H^{n}(G;A)\to H^{n+1}_{\pi_{1}(BG)}(BG;\Gamma)\to\cdots.
 \end{equation}
\end{remark}

\begin{tabsection}
 From the long exact sequence above, certain morphisms will turn out to be
 particularly important. We give them a distinguished name.
\end{tabsection}

\begin{definition}
 The morphisms $\varepsilon^{n}\from H^{n}(G;\Gamma)\to H^{n}(G;\fa)$, induced
 by the inclusion $\Gamma\hookrightarrow \fa $ will be called
 \emph{characteristic morphisms} in the sequel.
\end{definition}

\begin{tabsection}
 Note that in the theory of flat characteristic classes there is also the
 notion of \emph{characteristic morphism} (see \cite[Section
 2.3]{Morita01Geometry-of-characteristic-classes}). Our characteristic morphism
 is a refinement of the one occurring there that is more sensitive to the
 underlying topological information (see Remark
 \ref{rem:flat_characteristic_classes}). In particular, the characteristic
 morphisms in \cite[Section 2.3]{Morita01Geometry-of-characteristic-classes}
 are likely to be injective \cite[Theorem
 2.22]{Morita01Geometry-of-characteristic-classes}. In contrast to this we note
 that from the vanishing of $H^{n}_{\vE}(G;\fa)\cong H^{n}(G;\fa)$ for compact
 Lie groups and quasi-complete $\fa$
 \cite{HochschildMostow62Cohomology-of-Lie-groups} and the contractibility of
 $BG$ for contractible $G$ we have the following vanishing result for the
 characteristic morphism in the topologically trivial situations:
\end{tabsection}

\begin{lemma}
 The characteristic morphisms vanish if either $G$ is contractible or if $G$ is
 a compact Lie group and $\fa$ is quasi-complete and locally convex.
\end{lemma}

\begin{tabsection}
 This suggest that the characteristic morphisms are likely to vanish. We will
 show in the sequel that this is often the case and, if not, is the source of
 interesting geometric structure in terms of flat characteristic classes (cf.\
 Remark \ref{rem:flat_characteristic_classes} and Section \ref{sec:examples}).
 What we will show in the remainder of this section is that it is appropriate
 to think of the characteristic morphisms as some kind of connecting morphisms.
\end{tabsection}

\begin{proposition}\label{prop:van_Est_cohomology_isomorphism_in_coefficients}
 Suppose $G$ is 1-connected. If $q\from \fa\to A=\fa/\Gamma$ denotes the
 quotient morphism, then
 \begin{equation}\label{eqn:quotient_iso}
  q_{*}\from H^{n}_{\vE}(G;\fa)\to H^{n}_{\vE}(G;A)
 \end{equation}
 is an isomorphism for $n\geq 1$.
\end{proposition}

\begin{proof}
 We first show surjectivity. Let $f\from G^{n}\to A$ be continuous and satisfy
 $d f=0$. By the dual Dold-Kan correspondence we may assume without loss of
 generality that $f(1,...,1)=0$. Since $G$ is 1-connected, there exists a
 unique continuous $\wt{f}\from G^{n}\to \fa$ such that $\wt{f}(1,...,1)=0$
 and $q \circ \wt{f}=f$. Thus $q \circ d \wt{f}=d (q \circ \wt{f})=d f=0$ and
 since $d \wt{f}$ is uniquely determined by $q \circ d \wt{f}=0$ and
 $d \wt{f}(1,...,1)=0$ is follows that $d \wt{f}=0$. Thus
 \eqref{eqn:quotient_iso} is surjective.
 
 Injectivity is argued similarly. If $q \circ \wt{f}=d b$ for some continuous
 $b\from G^{n-1}\to A$, then we can lift $b$ to some continuous
 $\wt{b}\from G^{n-1}\to \fa$. Making the appropriate assumptions on the values
 in $(1,...,1)$, one can adjust things so that $d \wt{b}=\wt{f}$ and conclude
 that \eqref{eqn:quotient_iso} is injective.
\end{proof}

\begin{corollary}
 If $G$ is 1-connected, then the natural morphism
 $H^{n}_{\vE}(G;A)\to H^{n}_{\loc}(G;A)$ fits into the long exact sequence
 \begin{equation*}
  \cdots\to H^{n-1}_{\loc}(G;A)\to H^{n}_{\pi_{1}(BG)}(BG;\Gamma)\to H^{n}_{\vE}(G;A)\to 
  H^{n}_{\loc}(G;A)\to H^{n+1}_{\pi_{1}(BG)}(BG;\Gamma)\to\cdots.
 \end{equation*}
 In particular, $H^{n}_{\vE}(G;A)\cong H^{n}_{\loc}(G;A)$ if $G$ is
 contractible (the latter is a specialisation of \cite[Theorem
 5.16]{Fuchssteiner11A-Spectral-Sequence-Connecting-Continuous-With-Locally-Continuous-Group-Cohomology}).
\end{corollary}

\begin{tabsection}
 We now work towards an interpretation of the long exact sequence for
 finite-dimensional Lie groups.
\end{tabsection}

\begin{remark}
 Suppose that $K\leq G$ is a closed subgroup. Then the inclusion
 $i\from K\to G$ induces a restriction morphism
 \begin{equation*}
  i^{*}\from H^{n}(G;A)\to H^{n}(K;A)
 \end{equation*}
 (given on $H^{n}_{\loc}$ by restricting cochains to $K$, whence the name). On
 the other hand, $G$ acts on the quotient $G/K$ by left multiplication and we
 set
 \begin{equation*}
  H^{n}_{\vE}((G,K);A):=H^{n}_{\vE}(G/K,A)\quad\text{ and }\quad H^{n}_{\loc}((G,K);A):=H^{n}_{\loc}(G/K;A).
 \end{equation*}
 Note that $H^{n}_{\vE}((G,K);A)$ and $H^{n}_{\loc}((G,K);A)$ are the relative
 versions of the van Est and the locally continuous cohomology (compare to the
 relative Lie algebra cohomology in Section
 \ref{sec:the_relation_to_relative_lie_algebra_cohomology}). Since the quotient
 map $p\from G\to G/K$ is $G$-equivariant it induces morphisms in cohomology
 \begin{equation*}
  p^{*}\from  H^{n}_{\loc}((G,K);A)\to H^{n}_{\loc}(G;A),
 \end{equation*}
 given on the cochain level by $f\mapsto f \circ (p\times\cdots\times p)$.
\end{remark}

\begin{proposition}\label{prop:van_Est_relative_cohomology_isomorphism_in_coefficients}
 Suppose $K\leq G$ is a closed subgroup such that $G/K$ is 1-connected. If
 $q\from \fa\to A=\fa/\Gamma$ denotes the quotient morphism, then
 \begin{equation*}
  q_{*}\from H^{n}_{\vE}((G,K);\fa)\to H^{n}_{\vE}((G,K);A)
 \end{equation*}
 is an isomorphism for $n\geq 1$.
\end{proposition}

\begin{proof}
 Replacing $G$ by $G/K$, the prof of Proposition
 \ref{prop:van_Est_cohomology_isomorphism_in_coefficients} carries over
 verbatim.
\end{proof}

\begin{proposition}\label{prop:relative_local_iso_van_Est}
 Suppose that $G$ is a finite-dimensional Lie group with finitely many
 components, that $K\leq G$ is a maximal compact subgroup and that $\fa$ is a
 quasi-complete locally convex space. Then
 $H^{n}_{\loc}(G;\fa)\cong H^{n}_{\loc}((G,K);A)$ for $n\geq 1$.
\end{proposition}

\begin{proof}
 We consider the following diagram of morphisms of cochain complexes, in which
 the maps are pre-compositions with $p\from G\to G/K$, post-compositions with
 $q\from \fa\to A$ and inclusions of locally continuous maps into continuous
 ones:
 \begin{equation*}
  \vcenter{  \xymatrix{
  & C_{\vE}^{n}(G,\fa)^{G}\ar[dr]^{\alpha_{2}}\\
  C_{\vE}^{n}(G/K,\fa)^{G} \ar[ur]^{\alpha_{1}}\ar[r]^{\alpha_{3}}\ar[dr]^{\alpha_{4}} & C_{\loc}^{n}(G/K,\fa)^{G} \ar[r]^{\beta_{1}} & C_{\loc}^{n}(G,\fa)^{G}\\
  & C^{n}_{\vE}(G/K,A)^{G}\ar[r]^{\alpha_{5}} &C_{\loc}^{n}(G/K,A)^{G}
  }}.
 \end{equation*}
 Now $\alpha_{1}$ is a quasi-isomorphism, i.e., it induces an isomorphism in
 cohomology, by \cite[Corollaire
 III.2.2]{Guichardet80Cohomologie-des-groupes-topologiques-et-des-algebres-de-Lie}.
 Moreover, $\alpha_{2}$ is a quasi-isomorphism by \cite[Proposition
 III.6]{FuchssteinerWockel11Topological-Group-Cohomology-with-Loop-Contractible-Coefficients}.
 The contractibility of $G/K$ also implies that $\alpha_{3}$ and $\alpha_{5}$
 are quasi-isomorphisms by \cite[Theorem
 3.16]{Fuchssteiner11A-Spectral-Sequence-Connecting-Continuous-With-Locally-Continuous-Group-Cohomology}.
 Consequently, $\beta_{1}$ is a quasi-isomorphism. In addition, $\alpha_{4}$
 induces an isomorphism in cohomology if $n\geq 1$ by Proposition
 \ref{prop:van_Est_relative_cohomology_isomorphism_in_coefficients} since $G/K$
 is contractible. This induces the desired isomorphisms
 $H^{n}_{\loc}(G;\fa)\cong H^{n}_{\loc}((G,K);A)$ for $n\geq 1$.
\end{proof}

\begin{remark}\label{rem:identifications_for_log_exact_sequence_restriction_inflation}
 We will denote the isomorphism from the preceding proposition by
 \begin{equation*}
  \psi^{n}\from H^{n}_{\loc}(G;\fa)\xrightarrow{~\cong~} H^{n}_{\loc}((G,K);A).
 \end{equation*}
 The same argument also shows that there is an isomorphism in the locally
 smooth cohomology, which we also denote by
 \begin{equation*}
  \psi^{n}\from H^{n}_{\loc,s}(G;\fa)\xrightarrow{\cong} H^{n}_{\loc,s}((G,K);A).
 \end{equation*}
 Note that $\psi^{n}$ is not implemented by a canonical morphism on the cochain
 level. However, we have the morphism
 $\beta_{1}=p^{*}\from C^{n}_{\loc}(G/K,\fa)^{G}\to C^{n}_{\loc}(G,\fa)^{G}$.
 The preceding proof shows that this also induces an isomorphism
 \begin{equation*}
  p^{*}\from H^{n}_{\loc}(G;\fa)\xrightarrow{\cong} H^{n}_{\loc}((G,K);\fa)
 \end{equation*}
 in cohomology for all $n\in\N_{0}$.
 
 If $i\from K\to G$ denotes the inclusion, then this is a homotopy equivalence,
 and same is true for the induced map of classifying spaces $Bi\from BK\to BG$.
 Thus the induced map in cohomology $Bi^{*}$ is an isomorphism and the
 commuting diagram
 \begin{equation*}
  \xymatrix{ 
  H^{n+1}_{\loc}(G;\Gamma)\ar[d]^{i^{*}}\ar[r]^(.45){\cong}& H^{n+1}_{\pi_{1}(BG)}(BG;\Gamma)\ar[d]^{Bi^{*}}\\
  H^{n+1}_{\loc}(K;\Gamma)\ar[r]^(.45){\cong}& H^{n+1}_{\pi_{1}(BK)}(BK;\Gamma)
  }
 \end{equation*}
 shows hat $i^{*}\from H^{n}_{\loc}(G;\Gamma)\to H^{n}_{\loc}(K,\Gamma)$ is an
 isomorphism.
 
 With respect to these identifications, the characteristic morphisms
 ${\varepsilon}^{n}\from H^{n}_{\loc}(G;\Gamma)\to H^{n}_{\loc}(G;\fa)$ induce
 morphisms
 $  {\tilde{\varepsilon}}^{n}\from H^{n}_{\loc}(K;A)\to H^{n+1}_{\loc}((G,K);A)$
 that make
 \begin{equation*}
  \vcenter{\xymatrix{
  H^{n}_{\loc}(K;A)\ar[rr]^{\tilde{\varepsilon}^{n}}\ar[d]^{\delta^{n}} && H^{n+1}_{\loc}((G,K);A)\\
  H^{n+1}_{\loc}(K;\Gamma)\ar[r]^{(i^{*})^{-1}} & H^{n+1}_{\loc}(G;\Gamma)\ar[r]^{\varepsilon^{n+1}} & H^{n+1}_{\loc}(G;\fa)\ar[u]_{\psi^{n+1}}
  }}
 \end{equation*}
 commute. We shall call the morphisms $\tilde{\varepsilon}^{n}$ also
 \emph{characteristic morphisms.}
\end{remark}

\begin{tabsection}
 The following proposition illustrates that one should think of the
 characteristic morphism as some kind of connecting homomorphism.
\end{tabsection}

\begin{proposition}\label{prop:characteristic_morphism}
 Suppose that $G$ is a finite-dimensional Lie group with finitely many
 components, that $K\leq G$ is a maximal compact subgroup, that $\fa$ is a
 quasi-complete locally convex $G$-module and that $\Gamma\leq \fa$ is a
 discrete submodule. Then the sequence
 \begin{equation}\label{eqn:log_exact_sequence_restriction_inflation}
  H^{1}_{\loc}((G,K);A)\xrightarrow{p^{*}} \cdots\xrightarrow{\tilde{\varepsilon}^{n-1}} H^{n}_{\loc}((G,K);A)\xrightarrow{p^{*}} H^{n}_{\loc}(G;A)\xrightarrow{i^{*}} H^{n}_{\loc}(K;A)\xrightarrow{\tilde{\varepsilon}^{n}} H^{n+1}_{\loc}((G,K);A)\to\cdots
 \end{equation}
 is exact.
\end{proposition}

\begin{proof}
 We first observe that
 \begin{equation*}
  \xymatrix{ H^{n}_{\loc}(G;\fa)\ar[r]& H^{n}_{\loc}(G;A) \ar[d]^{i^{*}} \ar[r]^{\delta^{n}} & H^{n+1}_{\loc}(G;\Gamma)\ar[d]^{i^{*}}\ar[r]&H^{n+1}_{\loc}(G;\fa)\\
  H^{n}_{\loc}(K;\fa)\ar[r]&  H^{n}_{\loc}(K;A)\ar[r]^{\delta^{n}}& H^{n+1}_{\loc}(K;\Gamma) \ar[r]& H^{n+1}_{\loc}(K;\fa)}
 \end{equation*}
 commutes and has exact rows. If $n\geq 1$, then we have
 $H^{n}_{\loc}(K;\fa)\cong H^{n}_{\vE}(K;\fa)=0$ by \cite[Corollary
 II.8]{FuchssteinerWockel11Topological-Group-Cohomology-with-Loop-Contractible-Coefficients}
 and \cite[Lemma
 IX.1.10]{BorelWallach00Continuous-cohomology-discrete-subgroups-and-representations-of-reductive-groups}
 and thus $\delta^{n}\from H^{n}_{\loc}(K;A)\to H^{n+1}_{\loc}(K;\Gamma)$ is an
 isomorphism.Thus the exactness of
 \eqref{eqn:log_exact_sequence_restriction_inflation} follows from the
 definition of $\tilde{\varepsilon}^{n}$.
\end{proof}

\section{The relation to relative Lie algebra cohomology} %
\label{sec:the_relation_to_relative_lie_algebra_cohomology}

\begin{tabsection}
 We now discuss the relation of topological group cohomology (in the guise of
 locally smooth group cohomology) to Lie algebra cohomology. Whilst the
 previous sections are results on the topological group cohomology that were
 derived more or less from their definitions, the perspective of Lie algebra
 cohomology will really bring new facets into the game. To this end, we will
 have to use the locally smooth model $H^{n}_{\loc,s}(G;A)$ quite
 intensively.\\
 
 Unless mentioned otherwise, $G$ will throughout this section be a
 finite-dimensional Lie group with finitely many components and $K$ will be a
 maximal compact subgroup. The coefficient module is always of the form
 $A=\fa/\Gamma$, where $\fa$ is a smooth, locally convex and quasi-complete
 $G$-module and $\Gamma$ is a discrete submodule. We will denote the
 corresponding quotient morphisms by $p\from G\to G/K$ and $q\from \fa\to A$
 and injections by $i\from K\to G$ and $j\from \Gamma\to\fa$. Moreover, $\fg$
 denotes the Lie algebra of $G$ and $\fk$ the Lie algebra of $K$. Note that $A$
 is then also a module for $K$ and that $\fa$ is also a module for $\fg$ and
 for $\fk$.\\
 
 We first recall some basic notions. Let $\fh\leq \fg$ be an arbitrary
 subalgebra. The relative Lie algebra cohomology $H^{n}_{\Lie}((\fg,\fh);\fa)$
 is the cohomology of the based and invariant cochains in the
 Chevalley-Eilenberg complex
 $C^{n}_{\CE}(\fg,\fa):=\Hom_{\R}(\Lambda^{n}\fg,\fa)$, i.e.,
 \begin{equation*}
  C_{\CE}^{n}((\fg,\fh),\fa):=\Hom_{\R}(\Lambda^{n}(\fg/\fh),\fa)^{\fh}\cong \{\omega \from \Lambda^{n}\fg\to\fa\mid i_{y}(\omega)=0 \text{ and }\theta_{y}(\omega)=0 \text{ for all } y\in\fh\},
 \end{equation*}
 where $i_{y}(\omega)(x_{1},...,x_{n-1}):=\omega(y,x_{1},...,x_{n-1})$ and
 \begin{equation*}
  \theta_{y}(\omega)(x_{1},...,x_{n}):=\sum_{i=1}^{n}\omega(x_{1},...,[x_{i},y],...,x_{n})+y.\omega(x_{1},...,x_{n})
 \end{equation*}
 with respect to the Chevalley-Eilenberg differential
 \begin{equation*}
  d_{\CE}\omega (x_{0},...,x_{n}):=\sum_{1\leq i\leq n}(-1)^{i}x_{i}.\omega(x_{0},...,\wh{x_{i}},...,x_{n})+\sum_{1\leq i < j \leq n}(-1)^{i+j}\omega([x_{i},x_{j}],x_{0},...,\wh{x_{i}},...,\wh{x_{j}},...,x_{n}).
 \end{equation*}
 (cf.\ \cite[Section
 I.1]{BorelWallach00Continuous-cohomology-discrete-subgroups-and-representations-of-reductive-groups},
 \cite[n\textsuperscript{o}
 II.3]{Guichardet80Cohomologie-des-groupes-topologiques-et-des-algebres-de-Lie}
 or \cite[Chapter
 X]{GreubHalperinVanstone76Connections-curvature-and-cohomology}). If
 $i\from \fh\to\fg$ denotes the inclusion, then we have a sequence of cochain
 complexes
 \begin{equation*}
  C_{\CE}^{n}((\fg,\fh),\fa) \hookrightarrow C_{\CE}^{n}(\fg,\fa) \xrightarrow{i^{*}} C^{n}_{\CE}(\fh,\fa).
 \end{equation*}
 This gives rise to a sequence in cohomology
 \begin{equation}\label{eqn:three_term_sequence_for_relative_lie_algebra_cohomology}
  H^{n}_{\Lie}((\fg,\fh);\fa)\xrightarrow{\kappa^{n}} H^{n}_{\Lie}(\fg;\fa)\xrightarrow{i^{*}}H^{n}_{\Lie}(\fh;\fa)
 \end{equation}
 which is of order two, but which is in general far from being exact. For
 instance, $H^{n}_{\Lie}((\fg,\fh);\fa)\to H^{n}_{\Lie}(\fg;\fa)$ vanishes
 frequently (see Section \eqref{sec:subalgebras_non_cohomologous_to_zero}). The
 sequence \eqref{eqn:three_term_sequence_for_relative_lie_algebra_cohomology} is
 much more a part of the spectral sequence
 \begin{equation*}
  E_{2}^{p,q}=H^{p}_{\Lie}(\fh;\R)\otimes H^{q}_{\Lie}((\fg,\fh);\fa)\Rightarrow H^{p+q}_{\Lie}(\fg;\fa)
 \end{equation*}
 \cite[Chapitre VI]{Koszul50Homologie-et-cohomologie-des-algebres-de-Lie},
 where the morphisms from
 \eqref{eqn:three_term_sequence_for_relative_lie_algebra_cohomology} occur as
 edge homomorphisms.
 
 In case that $\fh=\fk$ and $K$ is not connected there is a subcomplex
 \begin{equation*}
  C_{\CE}^{n}((\fg,K),\fa):=\Hom(\Lambda^{n}\fg/\fk,\fa)^{K}
 \end{equation*}
 of $C_{\CE}^{n}((\fg,\fk),\fa)$, whose cohomology we denote by
 $H^{n}_{\Lie}((\fg,K);\fa)$. We clearly have an induced sequence in cohomology
 \begin{equation*}
  H^{n}_{\Lie}((\fg,K);\fa)\xrightarrow{\kappa^{n}} H^{n}_{\Lie}(\fg;\fa)\xrightarrow{i^{*}}H^{n}_{\Lie}(\fk;\fa)
 \end{equation*}
 and $C_{\CE}^{n}((\fg,K),\fa)= C_{\CE}^{n}((\fg,\fk),\fa)$ if $K$ is
 connected. \\
 
 We now introduce the differentiation homomorphism from locally smooth to Lie
 algebra cohomology.
\end{tabsection}

\begin{remark}(cf.\
 \cite[Section V.2]{Neeb06Towards-a-Lie-theory-of-locally-convex-groups},
 \cite[n\textsuperscript{o}
 III.7.3]{Guichardet80Cohomologie-des-groupes-topologiques-et-des-algebres-de-Lie})
 We want to differentiate in the identity, so we first identify
 $C^{n}_{\loc,s}(G,A)^{G}$ via $f\mapsto F$ with
 $F(g_{1},...,g_{n}):=F(1,g_{1},g_{1} g_{2},...,g_{1}\cdots g_{n})$ with
 \begin{equation*}
  \wt{C}_{\loc,s}^{n}(G,A):=\{f\from G^{n}\to A\mid f\text{ is smooth on some identity neighbourhood}\}.
 \end{equation*}
 (cf.\ \cite[Appendix
 B]{Neeb04Abelian-extensions-of-infinite-dimensional-Lie-groups}). It will be
 convenient to work with normalised cochains, so we set
 \begin{equation*}
  \wt{C}_{\loc,s}^{0,n}(G,A):=\{f\in \wt{C}_{\loc,s}^{n}(G,A)\mid
  f(g_{1},...,g_{n})=0 \text{ if }g_{i}=1\text{ for some }i\}.
 \end{equation*}
 and observe that the inclusion
 $\wt{C}_{\loc,s}^{0,n}(G,A)\hookrightarrow \wt{C}_{\loc,s}^{n}(G,A)$ is a
 quasi-isomorphism by the dual Dold-Kan correspondence.
 
 Now suppose that $M$ is a manifold and that $f\from M^{n}\to A$ is smooth. For
 $v_{n}\in T_{m_{n}}M_{n}$ we then set
 \begin{equation*}
  \partial_{n}({v_{n}})f(m_{n})\from M^{n-1}\to \fa,\quad 
  (m_{1},...,m_{n-1})\mapsto \delta(f)(0_{m_{1}},...,0_{m_{n-1}},v_{k}),
 \end{equation*}
 where $\delta(f):=f^{*}(\omega_{\op{MC}})$ is the left logarithmic derivative
 of $f$ (equivalently the pull-back of the Maurer-Cartan form
 $\omega_{\op{MC}}$ of $A$). With this we now set
 \begin{equation*}
  D^{n}\from \wt{C}_{\loc,s}^{n}(G,A)\to C_{\CE}^{n}(\fg,\fa),\quad D^{n}(f)(v_{1},...,v_{n}):= \sum_{\sigma\in S_{n}} \op{sgn}(\sigma)\partial_{1}(v_{\sigma( 1)})\cdots \partial_{n}(v_{\sigma (n)})f(1,...,1),
 \end{equation*}
 where $M=U$ for $U\se G$ an identity neighbourhood such that
 $\left.f\right|_{U\times \cdots U}$ is smooth and
 $v_{1},...,v_{n}\in\fg=T_{1}U$. This induces for $n\geq 1$ (and if $A=\fa$
 also for $n=0$) a morphism in cohomology
 \begin{equation*}
  D^{n}\from H^{n}_{\loc,s}(G;A)\to H^{n}_{\Lie}(\fg;\fa)
 \end{equation*}
 (see \cite[Theorem
 V.2.6]{Neeb06Towards-a-Lie-theory-of-locally-convex-groups}, \cite[Appendix
 B]{Neeb04Abelian-extensions-of-infinite-dimensional-Lie-groups},
 \cite[n\textsuperscript{o}
 III.7.3]{Guichardet80Cohomologie-des-groupes-topologiques-et-des-algebres-de-Lie}
 or \cite[Appendix I]{EstKorthagen64Non-enlargible-Lie-algebras}). The kernel
 of $D^{n}$ are precisely those cohomology classes which possess
 representatives by cochains in $\wt{C}_{\loc,s}^{n}(G,A)$ that are constant on
 some identity neighbourhood. These are the \emph{flat} classes in the locally
 smooth cohomology. They comprise precisely the image of the morphism
 \begin{equation*}
(A^{\delta}\to A)_{*} \from H^{n}_{\loc,s}(G;A^{\delta})\to H^{n}_{\loc,s}(G;A)
 \end{equation*}
 \cite[Remark
 V.14]{WagemannWockel13A-Cocycle-Model-for-Topological-and-Lie-Group-Cohomology}.
 We thus obtain an exact sequence
 \begin{equation*}
  \xymatrix{
  H^{n}_{\Top}(BG;A^{\delta})\ar[r]^{\flat^{n}} & H^{n}_{\loc,s}(G;A) \ar[r]^{D^{n}} & H^{n}_{\Lie}(\fg;\fa).
  }
 \end{equation*}
\end{remark}

\begin{lemma}
 If we set
 \begin{multline*}
  \wt{C}_{\loc,s}((G,K),A):=\{f\in \wt{C}^{n}_{\loc,s}(G,A)\mid
  f(g_{1},...,g_{n})=k_{0}.f(k_{0}^{-1}g_{1}k_{1},...,k_{n-1}^{-1}g_{n}k_{n})\\\text{ for all }g_{1},...,g_{n}\in G,k_{0},...,k_{n}\in K\},
 \end{multline*}
 then $\alpha\from C^{n}_{\loc,s}(G/K,A)^{G}\to \wt{C}^{n}_{\loc,s}((G,K),A) $,
 $f\mapsto F$ with
 $F(g_{1},...,g_{n}):=f(p(1),p(g_{1}),p(g_{1} g_{2}),...,p(g_{1}\cdots g_{n}))$
 is an isomorphism of cochain complexes. Moreover, $D^{n}$ maps the subcomplex
 \begin{equation*}
  \wt{C}_{\loc,s}^{0,n}((G,K),A):=\wt{C}_{\loc,s}^{n}((G,K),A)\cap \wt{C}_{\loc,s}^{0,n}(G,A)
 \end{equation*}
 to $C^{n}_{\CE}((\fg,K);\fa)$ and induces a morphism
 \begin{equation*}
  D^{n}\from H^{n}_{\loc,s}((G,K);A)\to H^{n}_{\Lie}((\fg,K);\fa)
 \end{equation*}
 in cohomology.
\end{lemma}

\begin{proof}
 Noting that we have a $K$-equivariant diffeomorphism $G\cong G/K\times K$ it
 is clear that
 \begin{multline*}
  \Map((G/K)^{n+1},A)\to \{f\in\Map(G^{n+1},A)\mid f(g_{0},...,g_{n})=f(g_{0}k_{0},...,g_{n}k_{n})\\\text{ for all }g_{0},...,g_{n}\in G,k_{0},...,k_{n}\in K\},\quad f\mapsto f \circ(p\times\cdots\times p)
 \end{multline*}
 is an isomorphism. A straight-forward calculation then shows that the
 composition with the isomorphism $f\mapsto F$ is $\alpha$ and that
 $\im(\alpha^{n})=\wt{C}^{n}_{\loc,s}((G,K),A)$.
 
 To verify $D^{n}(\wt{C}_{\loc,s}^{0,n}((G,K);A))\se C^{n}_{\CE}((\fg,K),\fa)$
 we observe that
 \begin{equation*}
  f(kg_{1}k^{-1},...,kg_{n}k^{-1})=k.f(g_{1},...,g_{n})\text{ for all }k\in K,g_{1},...,g_{n}\in G
 \end{equation*}
 implies $k.(D^{n}f)=D^{n}f$ for all $k\in K$. Moreover,
 \begin{equation*}
  f(k,g_{1},...,g_{n-1})=k^{-1}.f(1,g_{1},...,g_{n-1})=0 \text{ for all }k\in K,g_{1},...,g_{n-1}\in G
 \end{equation*}
 implies that $D^{n}f(y,x_{1},...,x_{n-1})$ vanishes if $y\in \fk$,
 $x_{1},...,x_{n-1}\in\fg$. To finish the proof we notice that the inclusion
 $\wt{C}_{\loc,s}^{0,n}((G,K);A)\hookrightarrow \wt{C}_{\loc,s}^{n}((G,K);A)$
 induces an isomorphism in cohomology, so $D^{n}$ is uniquely determined by its
 values on the subcomplex $\wt{C}_{\loc,s}^{0,n}((G,K);A)$.
\end{proof}

\begin{tabsection}
 The following isomorphism is sometimes also called the \emph{van Est}
 isomorphism. Note that $H^{n}_{\Lie}((\fg,K);\fa)=H^{n}_{\Lie}((\fg,\fk);\fa)$
 if $K$ is connected.
\end{tabsection}

\begin{theorem}\label{thm:vanEstIsomorphism}
 Under the hypothesis from the beginning of this section the morphism
 \begin{equation*}
  D^{n}\from  H^{n}_{\loc,s}((G,K);A)\to H^{n}_{\Lie}((\fg,K);\fa)
 \end{equation*}
 is an isomorphism.
\end{theorem}

\begin{proof}
 Consider the diagram of morphisms of cochain complexes
 \begin{equation*}
  \xymatrix{
  C_{\vE,s}^{n}(G/K,\fa)^{G}\ar[rr]^{(f\mapsto q \circ f)} \ar[drr]^{D^{n}} && \ar[d]^{D^{n}} C^{n}_{\vE,s}(G/K,A)^{G} \lhook\mkern-7mu \ar[r]
  & C^{n}_{\loc,s}((G,K),A)^{G}\ar[dl]^{D^{n}}\\
  & & C^{n}_{\CE}((\fg,K),\fa)
  }
 \end{equation*}
 which obviously commutes. Now
 $D^{n}\from H^{n}_{\vE,s}(G/K;\fa)\to H^{n}_{\Lie}((\fg,\fk);\fa)$ is an
 isomorphism by \cite[Corollaire III.7.2 and n\textsuperscript{o}
 III.7.3]{Guichardet80Cohomologie-des-groupes-topologiques-et-des-algebres-de-Lie}
 and the morphisms in the top row are quasi-isomorphisms by Proposition
 \ref{prop:van_Est_relative_cohomology_isomorphism_in_coefficients} and
 \cite[Section
 7]{Fuchssteiner11A-Spectral-Sequence-Connecting-Continuous-With-Locally-Continuous-Group-Cohomology}.
 This shows the claim.
\end{proof}

\begin{tabsection}
 It is this isomorphism that will enable us to access the cohomology groups
 $H^{n}(G;A)$ (mostly in the model $H^{n}_{\loc,s}(G;A)$). This is mostly
 because it connects to the well-understood algebraic picture by the following,
 obvious fact.
\end{tabsection}

\begin{Proposition}
 The diagram
 \begin{equation*}
  \xymatrix{
  H^{n}_{\loc,s}((G,K);A)\ar[r]^(.55){p^{*}}\ar[d]^{D^{n}} & H^{n}_{\loc,s}(G;A)\ar[r]^{i^{*}} \ar[d]^{D^{n}} & H^{n}_{\loc,s}(K;A) \ar[d]^{D^{n}} \\
  H^{n}_{\Lie}((\fg,K);\fa)\ar[r]^(.55){\kappa^{n}} & H^{n}_{\Lie}(\fg;\fa) \ar[r]^{i^{*}} & H^{n}_{\Lie}(\fk;\fa)
  }
 \end{equation*}
 commutes. In particular, the composed morphisms
 \begin{equation*}
  \xymatrix{ 
  & & H^{n}_{\loc,s}(G;A)\ar[dr]^{D^{n}} \\
  H^{n-1}_{\loc,s}(K;A)\ar[r]^(.45){\tilde{\varepsilon}^{n}} & H^{n}_{\loc,s}((G,K);A)\ar[ur]^{p^{*}}\ar[dr]^{D^{n}} & & H^{n}_{\Lie}(\fg;\fa) \\
  & & H^{n}_{\Lie}((\fg,K);\fa)\ar[ur]^{\kappa^{n}}
  }
 \end{equation*}
 vanish.
\end{Proposition}

\begin{Corollary}\label{cor:vanishing_of_characteristic_homomorphism}
 If ${\kappa^{n}}$ is injective, then the characteristic morphisms
 ${\varepsilon}^{n}$ and $\tilde{\varepsilon}^{n}$ vanish.
\end{Corollary}

\begin{remark}
 The fact that $D^{n} \circ {p^{*}} \circ  \tilde{\varepsilon}^{n}$ vanishes
 can also be understood by considering the commuting diagram
 \begin{equation*}
  \vcenter{  \xymatrix{
  H^{n}_{\loc,s}(G;\Gamma)\ar[rr]^{j_{*}=\varepsilon^{n}} && H^{n}_{\loc,s}(G;\fa)\ar[rr]^{q_{*}}\ar[d]^{D^{n}} && H^{n}_{\loc,s}(G;A)\ar[d]^{D^{n}}\\
  && H^{n}_{\Lie}(\fg;\fa) \ar@{=}[rr] && H^{n}_{\Lie}(\fg;\fa)
  }}.
 \end{equation*}
 The identifications from Remark
 \ref{rem:identifications_for_log_exact_sequence_restriction_inflation} turn
 $D^{n} \circ {p^{*}} \circ  \tilde{\varepsilon}^{n}$ into
 $D^{n}\circ {\varepsilon}^{n} =D^{n} \circ j_{*}$. But the image of $j_{*}$
 consists, on the cochain level, of maps that vanish on an identity
 neighbourhood, so that all derivatives of these maps vanish in the identity.
 Consequently, $D^{n} \circ j_{*}$ vanishes. However, the fact that
 $D^{n} \circ {p^{*}} \circ  \tilde{\varepsilon}^{n}$ vanishes is not that
 important (let alone obvious), it is the conjunction with the fact that it
 factors as $\kappa^{n}\circ D^{n}\circ \wt{\varepsilon}^{n}$ that will be
 important.
\end{remark}

\begin{tabsection}
 We end this section with the following very convenient relation between the
 locally smooth Lie group cohomology, the abstract group cohomology and the Lie
 algebra cohomology.
\end{tabsection}

\begin{theorem}\label{thm:disc}
 Suppose $G$ is a Lie group with finitely many components and $\fa$ is a
 quasi-complete $G$-module on which $G_{0}$ acts trivially. Let
 $\xi^{n}\from H^{n}_{\loc,s}(G;\fa)\to H^{n}_{\loc,s}(G^{\delta};\fa^{\delta})=H^{n}_{\op{gp}}(G;\fa)$
 be induced by mapping locally smooth cochains to abstract cochains in the bar
 complex. Then the diagram
 \begin{equation}\label{eqn:disc1}
  \vcenter{\xymatrix{
  H^{n}_{\Top}(BG;\fa)\ar[r]^{\zeta^{n}} \ar[dd]^{(G^{\delta}\to G)^{*}} &H^{n}_{\loc,s}(G;\fa^{\delta}) \ar[dd]^{(G^{\delta}\to G)^{*}} \ar[dr]^{(\fa^{\delta}\to\fa)_{*}} & &H^{n}_{\loc,s}((G,K);\fa)\ar[d]^{D^{n}}\ar[dl]_{p^{*}}\\
  &&H^{n}_{\loc,s}(G;\fa)\ar[dl]_{\xi^{n}}\ar[dr]^{D^{n}} & H^{n}_{\Lie}((\fg,K);\fa)\ar[d]^{\kappa^{n}}\\        
  H^{n}_{\Top}(BG^{\delta};\fa)   \ar[r]^{\zeta^{n}}  & H^{n}_{\loc,s}(G^{\delta};\fa^{\delta})   &&H^{n}_{\Lie}(\fg;\fa)
  }}
 \end{equation}
 commutes and the sequence
 \begin{equation}\label{eqn:disc2}
  H^{n}_{\loc,s}(G;\fa^{\delta})\xrightarrow{(\fa^{\delta}\to\fa)_{*}} H^{n}_{\loc,s}(G;\fa)\xrightarrow{D^{n}} H^{n}_{\Lie}(\fg;\fa)
 \end{equation}
 is exact. In particular,
 $D^{n}\from H^{n}_{\loc,s}(G;\fa)\to H^{n}_{\Lie}(\fg;\fa)$ factors through
 $\kappa^{n}\from H^{n}_{\Lie}((\fg,K);\fa)\to H^{n}_{\Lie}(\fg;\fa)$ and
 isomorphisms.
\end{theorem}

\begin{proof}
 The left rectangle commutes by Proposition
 \ref{prop:isomorphism_btw_SM_and_loc_is_natural_in_first_argument} and the
 triangles by the definition of the morphisms on the cochain level. As already
 observed, the sequence \eqref{eqn:disc2} is exact by \cite[Remark
 V.14]{WagemannWockel13A-Cocycle-Model-for-Topological-and-Lie-Group-Cohomology}.
\end{proof}

\begin{corollary}\label{cor:disc1a}
 Suppose the hypothesis of Theorem \ref{thm:disc} hold. If, in addition,
 $\kappa^{n}\from H^{n}_{\Lie}((\fg,K);\fa)\to H^{n}_{\Lie}(\fg;\fa)$ is
 injective, then $H^{n}_{\Top}(BG;\fa)\to H^{n}_{\Top}(BG^{\delta};\fa)$
 vanishes.
\end{corollary}

\begin{corollary}\label{cor:disc2}
 Suppose the hypothesis of Theorem \ref{thm:disc} hold. If, in addition, $G$
 admits a cocompact lattice, then
 \begin{equation*}
  \im\left(H^{n}_{\Top}(BG;\fa)\to H^{n}_{\Top}(BG^{\delta};\fa)\right)\cong \ker(H^{n}_{\Lie}((\fg,K);\fa)\to H^{n}_{\Lie}(\fg;\fa)).
 \end{equation*}
 This happens for instance if $G$ is semi-simple.
\end{corollary}

\begin{proof}
 If $G$ admits a cocompact lattice, then the $\xi^{n}$ is injective by
 \cite[16\textsuperscript{o}
 Th\'eor\`eme]{Blanc85Cohomologie-differentiable-et-changement-de-groupes.} (if
 we use the isomorphism $H^{n}_{\vE,s}(G;\fa)\to H^{n}_{\loc,s}(G;\fa)$ induced
 by $C^{n}_{\vE,s}(G,\fa)\hookrightarrow C^{n}_{\loc,s}(G,\fa)$ to pull back
 $\xi^{n}$ to $H^{n}_{\vE}(G;\fa)$). Thus $\xi^{n}$ maps
 $\ker (D^{n})=\im((\fa^{\delta}\to\fa)_{*})$ isomorphically onto
 $\im(\xi^{n} \circ (\fa^{\delta}\to\fa)_{*})=\im((G^{\delta}\to G)^{*})$. That
 semi-simple Lie groups admit cocompact lattices is \cite[Theorem
 C]{Borel63Compact-Clifford-Klein-forms-of-symmetric-spaces}.
\end{proof}

\begin{remark}\label{rem:flat_characteristic_classes}
 We now interpret diagram \eqref{eqn:disc1} of Theorem \ref{thm:disc} in terms
 of flat characteristic classes. Recall that a flat characteristic class is an
 element in $H^{n}_{\Top}(BG^{\delta};\R)$ (or also
 $H^{n}_{\Top}(BG^{\delta};\Z)$) \cite[Chapter
 9]{Dupont78Curvature-and-characteristic-classes}. Note also that $\xi^{n}$ is
 called characteristic morphism in the theory of flat characteristic classes
 (if one identifies $H^{n}_{\loc,s}(G,\R)$ with $H^{n}_{\Lie}((\fg,K);\R)$ via
 the van Est isomorphism) \cite[Section
 2.3]{Morita01Geometry-of-characteristic-classes}.
 
 Then the image of
 $ \flat^{n}\from H^{n}_{\Top}(BG;\R^{\delta})\to H^{n}_{\loc,s}(G;\R)$
 consists of those cohomology classes that are represented by locally smooth
 cochains that vanish on some identity neighbourhood. These are precisely the
 \emph{flat} cohomology classes in $H^{n}_{\loc,s}(G;\R)$ in the sense that the
 associated Lie algebra cohomology class vanishes (if $n=2$, then the flat
 classes in $H^{2}_{\loc,s}(G;\R)$ are precisely those classes that are
 represented by a flat principal bundle $\R\to \wh{G}\to G$
 \cite{Neeb02Central-extensions-of-infinite-dimensional-Lie-groups}). From
 \eqref{eqn:disc1} it thus follows that the flat classes in
 $H^{n}_{\loc,s}(G;\R)$ map under $\xi^{n}$ to flat characteristic classes.
 
 The relation to our characteristic morphism $\varepsilon^{n}$ is given by
 the diagram
 \begin{equation*}
  \vcenter{ \xymatrix{
  H^{n}_{\Top}(BG;\Z)\ar[dr]^{\flat^{n}}_{\cong} \ar[rrr]^{j_{*}}\ar[ddd]^{(G^{\delta}\to G)^{*}} &&&
  H^{n}_{\Top}(BG;\R)\ar[dl]_{\flat^{n}} \ar[ddd]^{(G^{\delta}\to G)^{*}}\\
  & H^{n}_{\loc,s}(G;\Z)\ar[d]^{(G^{\delta}\to G)^{*}}\ar[r]^{j_{*}=\varepsilon^{n}} & H^{n}_{\loc,s}(G;\R)\ar[d]^{(G^{\delta}\to G)^{*}}\\
  & H^{n}_{\loc,s}(G^{\delta};\Z)\ar[r]^{j_{*}}& H^{n}_{\loc,s}(G^{\delta};\R)\\
  H^{n}_{\Top}(BG^{\delta};\Z)\ar[ur]^{\flat^{n}}_{\cong}\ar[rrr]^{j_{*}} &&&H^{n}_{\Top}(BG^{\delta};\R)\ar[ul]_{\flat^{n}}^{\cong}
  }},
 \end{equation*}
 which commutes by the naturality of the involved morphisms. Note that
 \begin{equation*}
  (G^{\delta}\to G)^{*}\from H^{n}_{\loc,s}(G;\Z)\cong H^{n}_{\Top}(BG;\Z)\to H^{n}_{\loc,s}(G^{\delta};\Z)\cong H^{n}_{\Top}(BG^{\delta};\Z)
 \end{equation*}
 is injective by \cite[Corollary
 1]{Milnor83On-the-homology-of-Lie-groups-made-discrete}. If we assume,
 moreover, that $G$ is semi-simple, then
 \begin{equation*}
  (G^{\delta}\to G)^{*}\from H^{n}_{\loc,s}(G,\R)\cong H^{n}_{\vE}(G,\R)\to H^{n}_{\loc,s}(G^{\delta},\R)= H^{n}_{\op{gp}}(G;\R)
 \end{equation*}
 is also injective (cf.\ Corollary \ref{cor:disc2}). Thus our characteristic
 morphism \emph{coincides} (on the image of $(G^{\delta}\to G)^{*}$) with
 \begin{equation*}
  j^{*}\from H^{n}(BG^{\delta};\Z)\to H^{n}(BG^{\delta};\R).
 \end{equation*}
\end{remark}

\section{Subalgebras non-cohomologous to zero} %
\label{sec:subalgebras_non_cohomologous_to_zero}

\begin{tabsection}
 In this section we will analyse under which conditions all characteristic
 morphisms vanish. The setting is the same as in Section
 \ref{sec:the_relation_to_relative_lie_algebra_cohomology}.
\end{tabsection}

\begin{definition}(cf.\
 \cite[Section
 X.5]{GreubHalperinVanstone76Connections-curvature-and-cohomology}) Let
 $\fh\leq\fg$ be a subalgebra. We say that $\fh\leq\fg$ is non-cohomologous to
 zero (shortly n.c.z.) for $\fa$ if
 \begin{equation*}
  \kappa^{n}\from H^{n}_{\Lie}((\fg,\fh);\fa)\to H^{n}_{\Lie}(\fg;\fa)
 \end{equation*}
 is injective for all $n\in \N_{0}$. If $\fh\leq\fg$ is n.c.z.\ for $\fa=\R$,
 then we simply say that $\fh\leq\fg$ is n.c.z. More generally, we say that the
 maximal compact subgroup $K\leq G$ s n.c.z.\ for $\fa$ if
 \begin{equation*}
  \kappa^{n}\from H^{n}_{\Lie}((\fg,K);\fa)\to H^{n}_{\Lie}(\fg;\fa)
 \end{equation*}
 is injective for all $n\in\N_{0}$ and shortly that $K\leq G$ is n.c.z.\ if it
 is so for $\fa=\R$.
\end{definition}

Note that for $G$ connected we have that $\fk\leq\fg$ is n.c.z.\ for $\fa$ if
and only if $K\leq G$ is n.c.z.

\begin{proposition}\label{prop:Hloc_as_abelian_group}
 If $K\leq G$ is n.c.z., then all characteristic morphisms
 $\tilde{\varepsilon}^{n}$ vanish and the long exact sequence from Proposition
 \ref{prop:characteristic_morphism} splits for each $n\geq 1$ into short exact
 sequences
 \begin{equation}\label{eqn:Hloc_as_abelian_group_short_exact_seq}
  0\to H^{n}_{\loc}((G,K);A)\xrightarrow{p^{*}} H^{n}_{\loc}(G;A)\xrightarrow{i^{*}} H^{n}_{\loc}(K;A) \to 0
 \end{equation}
 In particular,
 $i^{*}\from H^{n}_{\loc}(G;A)\to H^{n}_{\loc}(K;A)$ is then surjective and
 $p^{*}\from H^{n}_{\loc}((G,K);A)\to H^{n}_{\loc}(G;A)$ is then injective for
 each $n\geq 1$. Moreover, we have in this case
 \begin{equation}\label{eqn:Hloc_as_abelian_group_split_exact_seq}
  H^{n}(G,A)\cong H^{n}_{\Lie}((\fg,K),\fa)\oplus H^{n+1}_{\pi_{1}(BG)}(BG;\Gamma)
 \end{equation}
 as abelian groups.
\end{proposition}

\begin{proof}
 From Corollary \ref{cor:vanishing_of_characteristic_homomorphism} we
 immediately deduce the splitting of the long exact sequence. By Theorem
 \ref{thm:vanEstIsomorphism} we have
 $H^{n}_{\loc}((G,K);A)\cong  H^{n}_{\Lie}((\fg,K),\fa)$ and as in Remark
 \ref{rem:identifications_for_log_exact_sequence_restriction_inflation} we see
 that $H^{n}_{\loc}(K;A)\cong H^{n+1}_{\pi_{1}(BG)}(BG;\Gamma)$. Since
 $H^{n}_{\Lie}((\fg,K),\fa)$ is a divisible abelian group the short exact
 sequence splits.
\end{proof}

\begin{tabsection}
 Determining whether $\fk\leq\fg$ is n.c.z.\ is particularly convenient
 for semi-simple $\fg$ by the Cartan decomposition.
\end{tabsection}

\begin{remark}\label{rem:cartan_decomposition}
 Suppose $\fg$ is semi-simple and $\fg=\fk\oplus\fp$ is a Cartan decomposition
 of $\fg$. Then we have $[\fk,\fk]\se\fk$, $[\fk,\fp]\se \fp$ and
 $[\fp,\fp]\se\fk$, and we denote by $\fg_{u}:=\fk\oplus i\fp$ the Lie algebra
 with the same underlying vector space and bracket defined for $x,y\in\fk$ and
 $v,w\in\fp$ by
 \begin{equation*} [x,y]_{u}:=[x,y], \quad
  [x,v]_{u}:=[x,v], \quad [v,w]_{u}:=-[v,w].
 \end{equation*}
 Then $\fg_{u}$ is a compact real form of the complexification $\fg_{\bC}$ and
 $\fk$ is a subalgebra of $\fg_{u}$. More precisely, $\fg_{u}$ is isomorphic to
 the subalgebra which is the direct sum of $\fk$ and $i\fp$ as subspaces of
 $\fg_{\bC}$ (see \cite[Section
 13.1+2]{HilgertNeeb12Structure-and-geometry-of-Lie-groups} or \cite[Section
 III.7]{Helgason78Differential-geometry-Lie-groups-and-symmetric-spaces} for
 details). Moreover, we have the identity $\fg=\fg_{u}$ as $\fk$-modules. If
 $\fa$ is the trivial $\fg$-module (also considered as tirival
 $\fg_{u}$-module), then this identity induces an isomorphism of $\fk$-modules
 \begin{equation}\label{eqn:isomorphism_of_relative_lie_algebra_cohomology_to_compact_dual}
  \{\omega\from \Lambda^{n}\fg\to\fa\mid i_{y}(\omega)=0 \text{ for all }y\in\fk\}=
  \{\omega\from \Lambda^{n}\fg_{u}\to\fa\mid i_{y}(\omega)=0 \text{ for all }y\in\fk\}
 \end{equation}
 and thus an isomorphism
 $\mu^{n}\from  H^{n}_{\Lie}((\fg,\fk);\fa)\xrightarrow{\cong} H^{n}_{\Lie}((\fg_{u},\fk);\fa)$.
\end{remark}

\begin{lemma}\label{lem:kappa_invective_iff_for_compact_dual}
 Suppose $\fg$ is real semi-simple, $\fg=\fk\oplus\fp$ is a Cartan
 decomposition of $\fg$ and set $\fg_{u}:=\fk\oplus i\fp$. If $\fa$ is the
 trivial $\fg$- and $\fg_{u}$-module, then
 $\kappa^{n}\from H^{n}_{\Lie}((\fg,\fk);\fa)\to H^{n}_{\Lie}(\fg;\fa)$ is
 injective if and only if
 $\kappa^{n}\from H^{n}_{\Lie}((\fg_{u},\fk);\fa)\to H^{n}_{\Lie}(\fg;\fa)$ is
 injective. In particular, $\fk\leq\fg$ is n.c.z.\ for $\fa$ if and only if
 $\fk\leq\fg_{u}$ is n.c.z.\ for $\fa$.
\end{lemma}

\begin{proof}
 The diagram
 \begin{equation*}
  \xymatrix{
  H^{n}_{\Lie}((\fg,\fk);\fa)\ar[r]^{\kappa^{n}}\ar[d]^{\cong} & H^{n}_{\Lie}(\fg;\fa)\lhook\mkern-7mu \ar[r]& H^{n}_{\Lie}(\fg;\fa)\otimes \bC \ar[r]^{\cong} & H^{n}_{\Lie}(\fg_{\bC},\fa_{\bC})\ar[d]^{\cong}\\
  H^{n}_{\Lie}((\fg_{u},\fk);\fa)\ar[r]^{\kappa^{n}} & H^{n}_{\Lie}(\fg_{u};\fa)\lhook\mkern-7mu \ar[r]& H^{n}_{\Lie}(\fg_{u};\fa)\otimes \bC \ar[r]^{\mu}_{\cong} & H^{n}_{\Lie}((\fg_{u})_{\bC},\fa_{\bC})\\
  }
 \end{equation*}
 commutes. This shows the claim.
\end{proof}

\begin{remark}\label{rem:compact_dual}
 The big advantage of $H^{n}_{\Lie}((\fg_{u},\fk);\fa)$ over
 $H^{n}_{\Lie}((\fg,\fk);\fa)$ is that $\fg_{u}$ is a compact Lie algebra, and
 thus $H^{n}_{\Lie}((\fg_{u},\fk);\fa)$ can be accessed as the de Rham
 cohomology of a \emph{compact} symmetric space. Let $\wt{G}_{u}$ be the simply
 connected Lie group with Lie algebra $\fg_{u}$. Then the embedding
 $\fk\hookrightarrow \fg_{u}$ induces an embedding
 $\wt{K}\hookrightarrow\wt{G}_{u}$. In particular, $\pi_{1}(K)$ embeds into
 $\wt{G}_{u}$ and we set $G_{u}:=\wt{G}_{u}/\pi_{1}(K)$. From this it is clear
 that $K=\wt{K}/\pi_{1}(K)$ embeds into $G_{u}$ and we will identify $K$ via
 this embedding with a subgroup of $G_{u}$. We call the pair $(G_{u},K)$ the
 \emph{dual pair} of $(G,K)$. Note that the property that $K$ embeds into
 $G_{u}$ determines uniquely the quotient of $\wt{G}_{u}$ that we have to take.
 If $G$ is linear, then another method for obtaining $G_{u}$ is to take a
 maximal compact subgroup of the complexification $G_{\bC}$ that contains $K$.
 However, the complexification exists in the semi-simple case if and only if
 $G$ is linear (cf.\ \cite[Proposition
 16.1.3]{HilgertNeeb12Structure-and-geometry-of-Lie-groups}).
 
 Now there is the canonical morphism
 $\nu\from H^{n}_{\Lie}((\fg_{u},\fk);\fa)\to H^{n}_{\dR}(G_{u}/K;\fa)\cong H^{n}_{\Top}(G_{u}/K;\fa)$
 that maps $\omega$ to the left invariant differential form on $G_{u}/K$ with
 value $\omega$ in $T_{e}(G_{u}/K)\cong \fg_{u}/\fk$. This is an isomorphism,
 for instance by \cite[Proposition
 XI.1.I]{GreubHalperinVanstone76Connections-curvature-and-cohomology} or
 \cite[Theorem 1.28]{FelixOpreaTanre08Algebraic-models-in-geometry}.
\end{remark}

\begin{example}\label{exmp:ncz}
 \begin{enumerate}
  \item \label{exmp:ncz1} If $G$ is compact, then $\fk\leq \fg$ is clearly
        n.c.z.\ for each $\fa$.
  \item \label{exmp:ncz2} Suppose $G$ is complex simple, considered as a simple
        real Lie group and $\fa$ is the trivial module. Then $\fk$ is a compact
        real form of $\fg$ and $\fk\oplus i\fk$ is a Cartan decomposition of
        $\fg$. Consequently, $\fg_{u}=\fk_{1}\oplus\fk_{2}$ with $\fk_{i}:=\fk$
        (we introduced the indices to distinguish the different copies of
        $\fk$). Then $\fk$ embeds as $\fk_{1}$ into $\fg_{u}$ and we have
        \begin{equation*}
         H^{n}_{\Lie}((\fk_{1}\oplus\fk_{2},\fk_{1});\fa)\cong H^{n}( \Hom(\Lambda^{\bullet}\fk_{2},\fa)^{\fk_{1}})\cong H^{n}( \Hom(\Lambda^{\bullet}\fk_{2},\fa))\cong H^{n}_{\Lie}(\fk_{2},\fa),
        \end{equation*}
        which embeds into $H^{n}_{\Lie}(\fk_{1}\oplus\fk_{2},\fa)$ by the
        K\"unneth Theorem.
  \item \label{exmp:ncz3} If $\fa$ is the trivial $\fg$-module, then
        $\fh\leq \fg$ is n.c.z.\ for $\fa$ if and only if
        $H^{*}_{\Lie}((\fg,\fh);\fa)$ is generated by $1$ and elements of odd
        degree \cite[Theorem
        X.10.19]{GreubHalperinVanstone76Connections-curvature-and-cohomology}.
        Moreover, this is the case if and only if
        $H^{*}_{\Lie}((\fg_{u},\fh);\fa)$ is n.c.z., which in turn is
        equivalent to $H^{n}_{\Top}(G_{u}/K;\fa)$ being generated by $1$ and
        elements of odd degree \cite[Theorem
        11.5.VI]{GreubHalperinVanstone76Connections-curvature-and-cohomology}.
 \end{enumerate}
\end{example}

\begin{example}[$G=\SL_{2q+1}$]\label{exmp:n.c.z.}
 Let $G=\SL_{p}$ with $p=2q+1\geq 3$ odd. Then $K=\SO_{p}$ and $G_{u}=\SU_{p}$.
 Then we have by \cite[Theorem
 III.6.7]{MimuraToda91Topology-of-Lie-groups.-I-II} that
 $H^{*}_{\Top}(\SU_{p}/\SO_{p};\R)$ is generated by $1$ and elements of odd
 degree. Thus $\fk\leq\fg$ is n.c.z.\ by Example \ref{exmp:ncz} \ref{exmp:ncz3}
 and we have the description of $H^{n}(\SL_{2q+1}(\R);U(1))$ from
 \eqref{eqn:Hloc_as_abelian_group_split_exact_seq}.
\end{example}

From Corollary \ref{cor:disc1a} we also obtain immediately

\begin{corollary}\label{cor:disc1b}
 Suppose $\fa$ is the trivial $\fg$-module. If $\fk\leq\fg$ is n.c.z.\ for
 $\fa$, then
 $H^{n}_{\Top}(BG;\fa^{\delta})\to H^{n}_{\Top}(BG^{\delta};\fa^{\delta})$
 vanishes.
\end{corollary}

\begin{tabsection}
 Note that Corollary \ref{cor:disc1b}, together with Example \ref{exmp:ncz}
 \ref{exmp:ncz1} and \ref{exmp:ncz2} give the well-known vanishing of
 $H^{n}_{\Top}(BG;\fa^{\delta})\to H^{n}_{\Top}(BG^{\delta};\fa^{\delta})$ in
 case that $G$ is either compact or $G$ is complex and semi-simple with
 finitely many components. The latter is usually proved directly via Chern-Weil
 theory, cf.\ \cite[Section 5.1]{Knudson01Homology-of-linear-groups} or
 \cite{Milnor83On-the-homology-of-Lie-groups-made-discrete} and the relation of
 the Chern-Weil homomorphism to the van Est cohomology
 \cite{Bott73On-the-Chern-Weil-homomorphism-and-the-continuous-cohomology-of-Lie-groups}.
 This is also implicit in here, as the next section shows.
\end{tabsection}

\section{Semi-simple Lie groups} %
\label{sec:semi_simple_lie_groups}

\begin{tabsection}
 In this section we will compute the characteristic homomorphism in terms of
 the Chern-Weil homomorphism of the compact dual of the symmetric space
 naturally associated to the non-compact symmetric space $G/K$. In particular,
 this will enable us to analyse cases in which not all characteristic
 homomorphisms vanish.
 
 The Setting is the same as in Section
 \ref{sec:the_relation_to_relative_lie_algebra_cohomology}, except that we
 assume, in addition, that $\fg$ is semi-simple and the induced $\fg$ module
 structure on $\fa$ is trivial\footnote{It would be desirable to have the
 results of this and the preceding section also for non-trivial coefficients.
 However, the techniques presented in this paper do not simply generalise to
 non-trivial coefficients for the following reasons: $A^{\delta}$ might not be
 a $G$-module any more; $\fa$ is not a $\fg_{u}$-module in a natural way; the
 Weil homomorphism is not well-defined, since the identity
 $W^{1}_{(\fg,\fv)}=d_{\CE}\circ \pi_{\fv}^{*}-\pi_{\fv}^{*}\circ d_{\CE}$ does
 not hold any more.}. Moreover, we choose and fix a Cartan decomposition
 $\fg=\fk\oplus \fp$ of $\fg$. We will use the notation from Remark
 \ref{rem:cartan_decomposition} and Remark \ref{rem:compact_dual}.
\end{tabsection}

\begin{theorem}\label{thm:characteristic_morphims_is_CW}
 Suppose $G$ is a connected finite-dimensional Lie group, that $\fg=L(G)$ is
 semi-simple and that $G$ acts trivially on the quasi-complete locally convex
 space $\fa$. If $f\from G_{u}/K\to BK$ is a classifying map for the principal
 $K$-bundle $G_{u}\to G_{u}/K$, then the diagram
 \begin{equation}\label{eqn:characteristic_morphims_is_CW}
  \vcenter{\xymatrix{
  H^{n}_{\loc,s}(G;\Gamma)\ar[dd]^{j_{*}=\varepsilon^{n}} 
  \ar[rr]^{(\flat^{n})^{-1}} &&
  H^{n}_{\Top}(BG;\Gamma)
  \ar[rr]^{Bi^{*}} &&
  H^{n}_{\Top}(BK;\Gamma) 
  \ar[d]^{j_{*}} &       
  \\
  &&&&   H_{\Top}^{n}(BK;\fa)
  \ar[d]^{f^{*}}
  \\
  H^{n}_{\loc,s}(G;\fa)
  \ar[r]^(.425){(p^{*})^{-1}} &
  H^{n}_{\loc,s}((G,K);\fa)
  \ar[r]^(.55){D^{n}} &
  H^{n}_{\Lie}((\fg,\fk);\fa)
  \ar[r]^{\mu^{n}} &
  H^{n}_{\Lie}((\fg_{u},\fk);\fa)
  \ar[r]^{\nu^{n}} &
  H_{\Top}^{n}(G_{u}/K;\fa)
  }}
 \end{equation}
 commutes and all horizontal morphisms are in fact isomorphisms.
\end{theorem}

\begin{tabsection}
 Note that the cohomology of $G_{u}/K$ and the morphisms
 $f^{*}\from H^{n}_{\Top}(BK;\fa)\to H^{n}_{\Top}(G_{u}/K;\fa)$ are well understood, for
 instance for $\fa=\R$ and simple $G$ (see for instance
 \cite{GreubHalperinVanstone76Connections-curvature-and-cohomology,Mimura95Homotopy-theory-of-Lie-groups,FelixOpreaTanre08Algebraic-models-in-geometry}).
 We will list some examples and applications of the theorem in the next
 section.
\end{tabsection}

\begin{proof}
 That all horizontal morphisms are isomorphisms has been argued in the previous
 sections. We will deduce the commutativity of the diagram by establishing a
 sequence of commuting diagrams that will give
 \eqref{eqn:characteristic_morphims_is_CW} in the end. We first consider
 \begin{equation}\label{eqn:proof_characteristic_morphims_is_CW_1}
  \vcenter{  \xymatrix{
  H^{n}_{\Top}(BK;\Gamma)\ar[d]_{j_{*}} \ar[r]^{(Bi^{*})^{-1}} &  H^{n}_{\Top}(BG;\Gamma)\ar[r]^{\flat^{n}}\ar[d]^{j_{*}}  & H^{n}_{\loc,s}(G;\Gamma)\ar[d]^{j_{*}}\\
  H^{n}_{\Top}(BK;\fa)  \ar[r]^{(Bi^{*})^{-1}} & H^{n}_{\Top}(BG;\fa) \ar[r]^{\flat^{n}} & H^{n}_{\loc,s}(G;\fa)
  }},
 \end{equation}
 which commutes by the naturality of $Bi^{*}$ and $\flat^{n}$. From this it
 follows that $j_{*}\from  H^{n}_{\loc,s}(G;\Gamma)\to H^{n}_{\loc,s}(G;\fa)$
 factors through $H^{n}_{\Top}(BG,\fa)$ and thus vanishes (by Hopf's Theorem)
 if $n$ is odd. Since
 $j_{*}\from H^{n}_{\Top}(BK,\Gamma)\to H^{n}_{\Top}(BK,\fa)$ vanishes for $n$
 odd for the same reason it suffices to show the commutativity of
 \eqref{eqn:characteristic_morphims_is_CW} if $n=2m$ is even.\\
 
 We now consider the algebraic Chern-Weil homomorphism
 \begin{equation*}
  \op{CW}^{m}_{(\fg,\fk)}\from \Hom_{\R}(S^{m}\fk;\fa)^{\fk}\to H^{2m}_{\Lie}((\fg,\fk);\fa),
 \end{equation*}
 which is defined as follows (cf.\
 \cite{GreubHalperinVanstone76Connections-curvature-and-cohomology}). Let
 $\fg=\fk\oplus\fp$ be the Cartan decomposition of $\fg$ and let
 $\pi_{\fp}\from \fg\to\fp$ and $\pi_{\fk}\from \fg\to \fk$ be the
 corresponding projections. Note that both are morphisms of $\fk$-modules. Then
 we set
 \begin{equation*}
  \op{CW}^{1}_{(\fg,\fk)}:=\frac{1}{2}\pi_{\fp}^{*}\circ d_{\CE} \circ \pi_{\fk}^{*}, \text{ i.e., }\quad \op{CW}^{1}_{(\fg,\fk)}(\lambda)(x,y):=\frac{1}{2}\lambda([\pi_{\fp}(x),\pi_{\fp}(y)])
 \end{equation*}
 Then $\op{CW}^{1}_{(\fg,\fk)}(\lambda)$ is $\fk$-invariant and clearly
 satisfies $i_{y}(\op{CW}^{1}_{(\fg,\fk)}(\lambda))=0$ for all $y\in \fk$. Thus
 we have
 \begin{equation}\label{eqn:proof_characteristic_morphims_is_CW_5}
  d_{\CE}\circ \op{CW}^{1}_{(\fg,\fk)}=\pi_{\fp}^{*} \circ d_{\CE} \circ \op{CW}^{1}_{(\fg,\fk)}= \pi_{\fp}^{*} \circ d_{\CE} \circ (d_{\CE}\circ \pi_{\fk}^{*}-\pi_{\fk}^{*}\circ d_{\CE})=- \pi_{\fp}^{*} \circ d_{\CE} \circ \pi_{\fk}^{*}\circ d_{\CE} =0
 \end{equation}
 where
 $\op{CW}^{1}_{(\fg,\fk)}=d_{\CE}\circ \pi_{\fk}^{*}-\pi_{\fk}^{*}\circ d_{\CE}$
 follows from
 \begin{equation*}
  \lambda(\pi_{\fk}([\pi_{\fp}(x),\pi_{\fp}(y)]))=
  \lambda(\pi_{\fk}([x-\pi_{\fk}(x),y-\pi_{\fk}(y)]))=\lambda(\pi_{\fk}([x,y]))-\lambda([\pi_{\fk}(x),\pi_{\fk}(y)])
 \end{equation*}
 and the last identity in \eqref{eqn:proof_characteristic_morphims_is_CW_5}
 follows from the fact that $d_{\CE}$ preserves $\im(\pi_{\fk}^{*})$. Thus
 $\op{CW}^{1}_{(\fg,\fk)}(\lambda)$ is closed and hence represents a class in
 $H^{2}_{\Lie}((\fg,\fk);\fa)$. Since $H^{\op{even}}_{\Lie}((\fg,\fk);\fa)$ is
 commutative the case $m=1$ determines a unique morphism of algebras
 $\op{CW}^{*}_{(\fg,\fk)}\from \Hom_{\R} (S^{*}\fk,\fa)^{\fk} \to H^{2*}_{\Lie}((\fg,\fk);\fa)$.\\
 
 The algebraic Chern-Weil homomorphism, together with the universal Chern-Weil
 isomorphism
 \begin{equation*}
  \wt{\op{CW}}^{m}\from \Hom(S^{m}\fk;\fa)^{\fk}\xrightarrow{\cong} H^{2m}_{\Top}(BK;\fa)
 \end{equation*}
 now give rise to a diagram
 \begin{equation}\label{eqn:proof_characteristic_morphims_is_CW_2}
  \vcenter{  \xymatrix{
  H^{2m}_{\Top}(BK;\fa) \ar[r]^{(Bi^{*})^{-1}}& H_{\Top}^{2m}(BG;\fa) \ar[r]^{\flat^{2m}} & H^{2m}_{\loc,s}(G,\fa)& H^{2m}_{\loc,s}((G,K);\fa) \ar[l]_{p^{*}} \\
  \Hom_{\R} (S^{m}\fk,\fa)^{\fk} \ar[rrr]^{\op{CW}^{m}_{(\fg,\fk)}}\ar[u]^{\wt{\op{CW}}^{m}} &&& H^{2m}_{\Lie}((\fg,\fk);\fa)\ar[u]_{(D^{2m})^{-1}}
  }}.
 \end{equation}
 We claim that this diagram commutes as well. To this end, let
 \begin{equation*}
  I^{2m}\from H^{2m}_{\Lie}((\fg,\fk);\fa)\to H^{2m}_{\vE,s}(G/K;\fa)
 \end{equation*}
 be the inverse of the van Est isomorphism, as described explicitly in
 \cite[n\textsuperscript{o}
 III.7.3]{Guichardet80Cohomologie-des-groupes-topologiques-et-des-algebres-de-Lie}
 or in \cite[Proposition
 1.5]{Dupont76Simplicial-de-Rham-cohomology-and-characteristic-classes-of-flat-bundles}.
 This has the property that
 \begin{equation}
  \xymatrix{H^{2m}_{\Lie}((\fg,\fk);\fa)\ar[r]^{I^{2m}} & H^{2m}_{\vE,s}(G/K;\fa)\ar[r]^{\cong}&H^{2m}_{\loc,s}((G,K);\fa)
  \ar@/^{2em}/[ll]_{D^{2m}}
  }
 \end{equation}
 commutes, where the unlabelled isomorphism is induced by the inclusion
 $C^{n}_{\vE,s}(G/K,\fa)\hookrightarrow C^{n}_{\loc,s}((G,K),\fa)$ of chain
 complexes (cf.\ \cite[Section
 7]{Fuchssteiner11A-Spectral-Sequence-Connecting-Continuous-With-Locally-Continuous-Group-Cohomology}).
 Now let $D\leq G$ be a cocompact lattice in $G$ (which exists by \cite[Theorem
 C]{Borel63Compact-Clifford-Klein-forms-of-symmetric-spaces}) and let
 $\iota\from D\to G$ denote the inclusion. Then the restriction
 $B \iota^{*}\from H^{n}_{\vE,s}(G;\fa)\to H^{n}_{\vE}(D;\fa)= H^{n}_{\op{gp}}(D,\fa)$
 is injective by \cite[15\textsuperscript{o}
 Th\'eor\`eme]{Blanc85Cohomologie-differentiable-et-changement-de-groupes.}.
 Thus we have the commuting diagram
 \begin{equation*}
  \vcenter{  \xymatrix{
  H_{\Top}^{2m}(BK;\fa) \ar[r]^{(Bi^{*})^{-1}} & H_{\Top}^{2m}(BG;\fa) \ar[r]^{\flat^{2m}}\ar[d]^{B \iota^{*}} & H^{2m}_{\loc,s}(G,\fa) \ar[d]^{B \iota^{*}} &&H^{2m}_{\loc,s}((G,K);\fa)\ar[ll]_{p^{*}} \\
  & H_{\Top}^{2m}(BD;\fa) \ar[r]^{\flat^{2m}} & H^{2m}_{\op{gp}}(D;\fa)
  & H^{2m}_{\vE,s}(G;\fa)\ar[l]_{B \iota^{*}} & H^{2m}_{\vE,s}(G/K;\fa)\ar[l]_{p^{*}}\ar[u]_{\cong}\\
  \Hom_{\R} (S^{m}\fk,\fa)^{\fk} \ar[rrrr]^{\op{CW}^{m}_{(\fg,\fk)}}\ar[uu]^{\wt{\op{CW}}^{m}} &&&& H^{2m}_{\Lie}((\fg,\fk);\fa)\ar[u]_{I^{2m}}
  }}.
 \end{equation*}
 Since $\flat^{2m}\from H^{2m}_{\Top}(BD;\fa)\to H^{2m}_{\op{gp}}(D;\fa)$ is
 just the isomorphism between the cohomology of the classifying space and the
 bar resolution for the discrete group $D$, the inner diagram commutes by
 \cite[Corollary 1.3, Proposition 1.5 and Lemma
 4.6]{Dupont76Simplicial-de-Rham-cohomology-and-characteristic-classes-of-flat-bundles}\footnote{One
 can also argue without using a cocompact lattice by \cite[Theorem
 9.12]{Dupont78Curvature-and-characteristic-classes}, but then one needs to
 assume that $\xi^{n}$ is injective, which follows in
 \cite[16\textsuperscript{o}
 Th\'eor\`eme]{Blanc85Cohomologie-differentiable-et-changement-de-groupes.}
 from the existence of a cocompact lattice.}. Since
 $B \iota^{*}\from H^{2m}_{\loc,s}(G;\fa)\to H^{2m}_{\op{gp}}(D;\fa)$ is
 injective we thus conclude that the outer diagram, and thus diagram
 \eqref{eqn:proof_characteristic_morphims_is_CW_2}, commutes.\\
 
 We now consider the algebraic Chern-Weil homomorphism
 $\op{CW}^{m}_{(\fg_{u},\fk)}$ for $\fg_{u}$ with respect to the decomposition
 $\fg_{u}=\fk\oplus i\fp$. Note that the underlying vector spaces of $\fg$ and
 $\fg_{u}$ are the \emph{same}, as well as the subspaces $\fp$ and $i\fp$.
 Since the cochains representing $\op{CW}^{m}_{(\fg,\fk)}(\lambda)$ and
 $\op{CW}^{m}_{(\fg_{u},\fk)}(\lambda)$ only depend on the projections onto
 $\fk$ and $i\fp$ (respectively $\fk$ and $\fp$) we conclude that the diagram
 \begin{equation}\label{eqn:proof_characteristic_morphims_is_CW_3}
  \vcenter{  \xymatrix{
  \Hom_{\R} (S^{n}\fk;\fa)^{\fk} \ar@{=}[d]\ar[rr]^{\op{CW}^{n}_{(\fg,\fk)}} && H^{2n}_{\Lie}((\fg,\fk);\fa)\ar[d]^{\mu}_{\cong}\\
  \Hom_{\R} (S^{n}\fk,\fa)^{\fk} \ar[rr]^{{  \op{CW}^{n}_{(\fg_{u},\fk)}}}  && H^{2n}_{\Lie}((\fg_{u},\fk);\fa)
  }}.
 \end{equation}
 commutes.
 
 Since the Chern-Weil homomorphism $\op{CW}_{\pi}$ for the bundle
 $\pi\from G_{u}\to G_{u}/ K$ factors through the universal Chern-Weil
 isomorphism and $f^{*}$ we obtain by \cite[Section
 8.26]{GreubHalperinVanstone76Connections-curvature-and-cohomology} the
 commuting diagram
 \begin{equation}\label{eqn:proof_characteristic_morphims_is_CW_4}
  \vcenter{  \xymatrix{
  \Hom_{\R} (S^{n}\fk,\fa)^{\fk} \ar[rrr]^{{  \op{CW}^{n}_{(\fg_{u},\fk)}}} \ar[d]^{\wt{\op{CW}}^{n}}_{\cong}\ar[drrr]^{\op{CW}_{\pi}^{n}} &&& H^{2n}_{\Lie}((\fg_{u},\fk);\fa)\ar[d]^{\nu}_{\cong}\\
  H^{2n}_{\Top}(BK;\fa)\ar[rrr]^{f^{*}} &&& H^{2n}_{\Top}(G_{u}/K;\fa)
  }}.
 \end{equation}
 If we paste the above diagrams
 \eqref{eqn:proof_characteristic_morphims_is_CW_1},
 \eqref{eqn:proof_characteristic_morphims_is_CW_2},
 \eqref{eqn:proof_characteristic_morphims_is_CW_3} and
 \eqref{eqn:proof_characteristic_morphims_is_CW_4}, then the outer diagram
 yields precisely \eqref{eqn:characteristic_morphims_is_CW}. This finishes the
 proof.
\end{proof}

\begin{remark}\label{rem:injectivity_of_j_star}
 Note that $j_{*}\from H^{n}_{\Top}(BK;\Gamma)\to H^{n}_{\Top}(BK;\fa)$ is very
 well-behaved in this particular case. If $H_{\op{odd}}(BK;\Gamma)$ is finitely
 generated and torsion free, then we get
 \begin{equation*}
  \xymatrix{
  H_{\Top}^{n}(BK;\Gamma)\ar[r]^(.4){\cong}\ar[d] & \Hom(H_{n}(BK);\Gamma)\ar[d]\\
  H_{\Top}^{n}(BK;\fa)\ar[r]^(.4){\cong} & \Hom(H_{n}(BK);\fa)
  }
 \end{equation*}
 from the Universal Coefficient Theorem. Thus $j_{*}$ is injective in this
 case. If, moreover, $\fa$ is separable and $\Gamma$ is countable, then it is
 free \cite[Remark 9.5
 (c)]{Neeb02Central-extensions-of-infinite-dimensional-Lie-groups}, and thus
 $\Hom(H^{n}(BK);\Gamma)$ injects into $\Hom(H^{n}(BK);\fa)$. If $\fa=\R^{n}$
 and $\Gamma$ is a lattice in $\R^{n}$, then so is $\Hom(H^{n}(BK);\Gamma)$ in
 $\Hom(H^{n}(BK);\R^{n})$ and a basis for $\Gamma$ then gives a basis for
 $\Hom(H^{n}(BK);\Gamma)$. All the above assumptions are in particular
 satisfied for $\fa=\R$, $\Gamma=\Z$ and
 $K=\op{U}_{q},\SU_{q}, \Sp_{q}, \SO_{q}$ (see \cite[Theorem
 5.5.10]{Spanier66Algebraic-topology}, \cite[Theorem
 16.17]{Switzer75Algebraic-topology---homotopy-and-homology}, \cite[Corollary
 III.3.11]{MimuraToda91Topology-of-Lie-groups.-I-II} and
 \cite{Brown82The-cohomology-of-Brm-SOn-and-Brm-On-with-integer-coefficients}).
\end{remark}

\section{Examples} %
\label{sec:examples}

\begin{tabsection}
 We will stick in this section to examples of simple linear Lie groups and the
 trivial coefficient modules $\Z$, $\R$ and $U(1)=\R/\Z$. We will calculate the
 characteristic homomorphisms
 $\varepsilon^{n}\from H^{n}_{\loc}(G;\Z)\to H^{n}_{\loc}(G;\R)$ via the
 commuting diagram
 \begin{equation*}
  \xymatrix{
  H^{n}_{\loc}(G;\Z)\ar[rr]^{\varepsilon^{n}=j_{*}}\ar[d]^{\cong} & & H^{n}_{\loc}(G;\R) \ar[d]^{\cong}\\
  H_{\Top}^{n}(BK;\Z)\ar[r]^{j_{*}} & H_{\Top}^{n}(BK;\R)\ar[r]^{f^{*}} & H_{\Top}^{n}(G_{u}/K;\R)
  }
 \end{equation*}
 from Theorem \ref{thm:characteristic_morphims_is_CW} (and identify
 $\varepsilon^{n}$ with $f^{*} \circ j_{*}$ by this). This will then give
 complete information on $H^{n}_{\loc}(G;U(1))$ (as abelian group) via the long
 exact sequence from Section \ref{sec:the_long_exact_sequence}.
 
 To this end we first recall the following facts from \cite[Section
 11.5]{GreubHalperinVanstone76Connections-curvature-and-cohomology} on the
 cohomology of a homogeneous space $G/H$ of a general compact connected Lie
 group $G$ with closed and connected subgroup $H$. Let $\pi\from G\to G/H$
 denote the corresponding principal $H$-bundle, and let $f\from G/H\to BH$ be a
 classifying map for $\pi$. Then there is an isomorphism
 \begin{equation*}
  \Phi\from   H_{\Top}^{*}(G/H;\R)\xrightarrow{\cong}  A_{\pi} \otimes \Lambda \wh{P}_{\pi}
 \end{equation*}
 of graded algebras, where
 $A_{\pi}:=\op{CW}^{*}_{\pi}(\Hom_{\R}(S^{*}\fh,\R)^{\fh})=\im(f^{*}\from H_{\Top}^{*}(BH;\R)\to H_{\Top}^{*}(G/H;\R))$
 is the image of the Chern-Weil homomorphism of $\pi$ and
 $\wh{P}_{\pi}:= P_{G}\cap \im(\pi^{*}\from H_{\Top}^{*}(G/H;\R)\to H_{\Top}^{*}(G;\R))$ for
 $P_{G}\leq H_{\Top}^{*}(G;\R)$ the graded subspace of primitive elements. Moreover,
 $\Phi$ makes the diagram
 \begin{equation*}
  \xymatrix{
  A_{\pi}\ar[r]  &  A_{\pi} \otimes \Lambda \wh{P}_{\pi}\ar[r] & H_{\Top}^{2*}(G;\R)\ar@{=}[d]\\
  \Hom_{\R}(S^{*}\fh,\R)^{\fh} \ar[u]_{\op{CW}_{\pi}^{*}} \ar[d]^{\wt{\op{CW}}^{*}} \ar[r] ^{\op{CW}_{\pi}^{*}} &  H_{\Top}^{2*}(G/H;\R)\ar@{=}[d]\ar[u]_{\Phi} \ar[r]^{\pi^{*}} & H_{\Top}^{2*}(G;\R)\\
  H_{\Top}^{2*}(BH;\R)\ar[r]^{f^{*}} &  H_{\Top}^{2*}(G/H;\R)
  }
 \end{equation*}
 commute.

\end{tabsection}

\begin{example}[$G=SL_{p}(\bC)$]
 Then $K=\SU_{p}$ and $G_{u}=\SU_{p}\times \SU_{p}$. By Example \ref{exmp:ncz}
 \ref{exmp:ncz2} we have that $\fk\leq \fg$ is n.c.z.\ for all $\fa$ and thus
 Proposition \ref{prop:Hloc_as_abelian_group} yields
 \begin{equation*}
  H^{n}(\SL_{p}(\bC),U(1))\cong H^{n}_{\Lie}((\mf{sl}_{p}(\bC),\mf{su}_{p}),\R)\oplus H_{\Top}^{n}(B\SU_{p};\Z)\cong H^{n}_{\Top}(\SU_{p};\R)\oplus H_{\Top}^{n}(B\SU_{p};\Z)
 \end{equation*}
 for each $n\geq 1$ (cf.\ Remarks \ref{rem:cartan_decomposition} and
 \ref{rem:compact_dual}), and the groups on the right are well-known (see for
 instance \cite[Corollary 1.86]{FelixOpreaTanre08Algebraic-models-in-geometry}
 and \cite[Corollary 4D.3]{Hatcher02Algebraic-topology}).
\end{example}

\begin{tabsection}
 We now run through some interesting and illustrative cases in which $G$ is a
 connected non-compact real form of a simple complex Lie group, $K$ is the
 maximal compact of $G$ and $G_{u}$ the maximal compact of the complexification
 $G_{\bC}$ (see \cite[Chapter
 X]{Helgason78Differential-geometry-Lie-groups-and-symmetric-spaces} for
 notation and details).
\end{tabsection}

\begin{example}[$G=\SL_{p}(\R)$]
 Then $K=\SO_{p}$ and $G_{u}=\SU_{p}$. The case $p=2q+1\geq 3$ has been treated
 in Example \ref{exmp:n.c.z.}. If $p=2q\geq 4$ is even, then
 \begin{equation*}
  H_{\Top}^{*}(BK;\R)\cong   \Hom_{\R}(S^{*}\mf{so}_{2q};	\R)^{\mf{so}_{2q}}=\Lambda(P_{1},...,P_{q-1},E_{q})
 \end{equation*}
 is generated by the Pontryagin classes $P_{1},...,P_{q-1}$, where $P_{i}$ is
 of degree $4i$, and the Euler class $E_{q}$ of degree $2q$. Moreover,
 $E_{q}^{2}=P_{q}$ \cite[Theorem 15.9]{MilnorStasheff74Characteristic-classes}.
 We now consider the kernel of $\op{CW}_{\pi}^{*}$. By \cite[Proposition
 10.6.III]{GreubHalperinVanstone76Connections-curvature-and-cohomology} it is
 generated (as an algebra without unit) by the image of
 $i^{*}\from \Hom_{\R}(S^{*}\mf{su}_{2q},\R)\to \Hom_{\R}(S^{*}\mf{so}_{2q},\R)$.
 From \cite[Example
 11.11.4]{GreubHalperinVanstone76Connections-curvature-and-cohomology} we get
 \begin{equation*}
  i^{*}(C_{i})= \begin{cases}
  0 & \text{ if }i \text{ odd}\\
  (-1)^{i/2}P_{i/2} & \text{ if }i \text{ even}
  \end{cases}
 \end{equation*}
 where $C_{2},...,C_{2q}$ denote the Chern classes ($C_{1}$ is missing since we
 consider $\SU_{p}$, not $\op{U}_{p}$).
 Thus we have $A_{\pi}=\R[E_{q}]/(E_{q}^{2})$ and by \cite[Proposition
 10.26.VII]{GreubHalperinVanstone76Connections-curvature-and-cohomology} we
 have that $\ker(i^{*})\cong\wh{P}_{\pi}$ is generated by the suspensions of
 the odd Chern classes $\wt{C}_{3},\wt{C}_{5},...,\wt{C}_{2q-1}$ with
 $\wt{C}_{i}\in H_{\Top}^{2i-1}(\SU_{p};\R)$. Consequently,
 \begin{equation*}
  H_{\Top}^{*}(\SU_{p}/\SO_{p};\R)\cong \big(\R[E_{q}]/(E_{q}^{2})\big) \otimes \Lambda (\wt{C}_{3},\wt{C}_{5},...,\wt{C}_{2q-1}).
 \end{equation*}
 In particular,
 $f^{*}\from H_{\Top}^{2q}(B\SO_{2q};\R)\to H_{\Top}^{2q}(\SU_{2q}/\SO_{2q};\R)$
 does not vanish on the Euler class. By Remark \ref{rem:injectivity_of_j_star},
 $j_{*}\from H_{\Top}^{n}(B\SO_{2q};\Z)\to H_{\Top}^{n}(B\SO_{2q};\R)$ is
 injective. Thus the characteristic homomorphism
 \begin{equation*}
  \varepsilon^{2q}\from H^{2q}(\SL_{2q}(\R);\Z)\to H^{2q}(\SL_{2q}(\R);\R)
 \end{equation*}
 does not vanish on the integral Euler class and
 $0\neq \varepsilon^{2q}(E_{q})\in H^{2q}(\SL_{2q}(\R);\R)$ is flat (cf.\
 Remark \ref{rem:flat_characteristic_classes}).
\end{example}

\begin{example}[$G=\SU^{*}_{2p}$]
 Then $K=\Sp_{p}$ and $G_{u}=\SU_{2p}$. From \cite[Theorem
 III.6.7]{MimuraToda91Topology-of-Lie-groups.-I-II} one sees that
 $H_{\Top}^{*}(G_{u}/K;\R)$ is generated by $1$ and elements of odd degree.
 Thus $\mf{sp}_{p}$ is n.c.z.\ in $\mf{su}_{2p}$ by Example \ref{exmp:ncz}
 \ref{exmp:ncz3}. Consequently, all characteristic homomorphisms
 \begin{equation*}
  \varepsilon^{n}\from H^{n}(\SU^{*}_{2p};\Z)\to H^{n}(\SU^{*}_{2p};\R)
 \end{equation*}
 vanish and we have the description of $H_{\Top}^{n}(\SU^{*}_{2p};U(1))$ from
 \eqref{eqn:Hloc_as_abelian_group_split_exact_seq}.
\end{example}

\begin{example}[$G=\Sp_{p}(\R)$]
 Then $K=\op{U}_{p}$ and $G_{u}=\Sp_{p}$. From \cite[Theorem III.6.9
 (1)]{MimuraToda91Topology-of-Lie-groups.-I-II} one sees that
 $H_{\Top}^{*}(G_{u}/K;\R)$ is evenly graded and that
 \begin{equation*}
  f^{*}\from H_{\Top}^{*}(B\op{U}_{p};\R)\to H_{\Top}^{*}(\Sp_{p}/\op{U}_{p};\R)
 \end{equation*}
 is surjective with kernel generated by the alternating products
 $\{\sum_{i+j=2k}(-1)^{i}C_{i}C_{j}\mid k\geq 1\}$ of the Chern classes. By
 Remark \ref{rem:injectivity_of_j_star},
 $j_{*}\from H_{\Top}^{n}(B\op{U}_{p};\Z)\to H_{\Top}^{n}(B\op{U}_{p};\R)$ is
 injective. Thus the characteristic homomorphism
 \begin{equation*}
  \varepsilon^{n}\from H^{n}(\Sp_{p}(\R);\Z) \to H^{n}(\Sp_{p}(\R);\R)
 \end{equation*}
 has as kernel precisely the integral linear combinations of the alternating
 products $\sum_{i+j=2k}(-1)^{i}C_{i}C_{j}$ of the \emph{integral} Chern
 classes (for $k\geq 1$). In particular,
 $0\neq \varepsilon^{2n}(C_{n})\in H^{2n}(\Sp_{p}(\R);\R)$ is flat (cf.\ Remark
 \ref{rem:flat_characteristic_classes}).
 
 With \cite[Theorem III.6.9 (2+3)]{MimuraToda91Topology-of-Lie-groups.-I-II}, a
 similar argument also applies to $G=\SU_{(p,q)}$ and $G=\Sp_{(p,q)}$.
\end{example}

\begin{remark}
 The results on the non-vanishing of the characteristic morphisms on the Euler
 class or the Chern classes are a refinement of some well-known identities in
 the abstract group cohomology $H^{n}_{\op{gp}}(G;\R)$ for $G=\SL_{p}(\R)$ and
 $G=\Sp_{n}(\R)$ (see
 \cite{Milnor58On-the-existence-of-a-connection-with-curvature-zero} and
 \cite[Section 9]{Dupont78Curvature-and-characteristic-classes}). What follows
 from re results of this paper is that these classes do not only live in
 $H^{n}_{\op{gp}}(G;\R)$, but that they lift to the topological group
 cohomology $H^{n}(G;\R)$.
\end{remark}

 We end this section with establishing the following stability result.

\begin{proposition}
 The natural homomorphisms $\SL_{p}(\bC)\to \SL_{p+1}(\bC)$ and
 $\SL_{p}(\R)\to \SL_{p+1}(\R)$ induce isomorphisms
 \begin{equation*}
  H^{n}(\SL_{p+1}(\bC);U(1))\xrightarrow{~\cong~} H^{n}(\SL_{p}(\bC);U(1))
  \quad\text{ and }\quad  H^{n}(\SL_{p+1}(\R);U(1))\xrightarrow{~\cong~} H^{n}(\SL_{p}(\R);U(1))
 \end{equation*}
 for sufficiently large $p$.
\end{proposition}

\begin{proof}
 From the descriptions of
 \begin{equation*}
  H^{n}_{\Top}(B \SO_{p};\Z)\cong H^{n}_{\Top}(B \SL_{p}(\R);\Z)\quad\text{ and }\quad   H^{n}_{\Top}(\SU_{p}/\SO_{p};\R)\cong H^{n}_{\vE}(\SL_{p}(\R);\R)
 \end{equation*}
 in
 \cite{Brown82The-cohomology-of-Brm-SOn-and-Brm-On-with-integer-coefficients}
 and \cite[Theorem III.6.7 (2)]{MimuraToda91Topology-of-Lie-groups.-I-II} on
 sees that $\SL_{p}(\R)\to \SL_{p+1}(\R)$ induces an isomorphism for
 sufficiently big $p$. Thus the same holds for $H^{n}(\SL_{p}(\R);U(1))$ by the
 long exact sequence \eqref{eqn:long_exact_sequence_2} from Remark
 \ref{rem:long_exact_sequence} and the Five Lemma. The argument for
 $\SL_{p}(\bC)$ is exactly the same (cf.\ \cite[Corollary III.3.11 and Theorem
 III.5.5]{MimuraToda91Topology-of-Lie-groups.-I-II}).
\end{proof}

\def\polhk#1{\setbox0=\hbox{#1}{\ooalign{\hidewidth
  \lower1.5ex\hbox{`}\hidewidth\crcr\unhbox0}}} \def\cprime{$'$}

\end{document}